\documentclass[11pt, notitlepage]{article}
\usepackage{amssymb,amsmath,comment}
\catcode`\@=11 \@addtoreset{equation}{section}
\def\thesection{\arabic{section}}

\def\theequation{\thesection.\arabic{equation}}
\catcode`\@=12
\usepackage{enumitem}
\usepackage{colortbl}%
\usepackage{a4wide}
\usepackage{parskip}

\newcommand{\ds} {\displaystyle}
\newcommand{\e}{\epsilon}

\newcommand{\al} {\alpha}
\newcommand{\ba} {\beta}

\newcommand{\de} {\delta}

\newcommand{\ra} {\rightarrow}

\newcommand{\De} {\Delta}

\newcommand{\noi} {\noindent}

\newcommand{\mb} {\mathbb}
\newcommand{\mc} {\mathcal}

\newcommand{\ld} {\langle}
\newcommand{\rd} {\rangle}
\newcommand{\I}{\int\limits_}

\usepackage[all]{xy}
\catcode`\@=11
\def\theequation{\@arabic{\c@section}.\@arabic{\c@equation}}
\catcode`\@=12

\def\QED{\hfill {$\square$}\goodbreak \medskip}

\newtheorem{Theorem}{Theorem}[section]
\newtheorem{Lemma}[Theorem]{Lemma}
\newtheorem{Proposition}[Theorem]{Proposition}
\newtheorem{Corollary}[Theorem]{Corollary}
\newtheorem{Remark}[Theorem]{Remark}
\newtheorem{Definition}[Theorem]{Definition}

\def\XXint#1#2#3{{\setbox0=\hbox{$#1{#2#3}{\int}$ }
		\vcenter{\hbox{$#2#3$ }}\kern-.6\wd0}}

\begin{document}
	{\vspace{0.01in}
		\title
		{Fractional Hamiltonian type system on $\mb R$ with critical growth nonlinearity }

			\author{ {\bf G.C. Anthal\footnote{	Department of Mathematics, Indian Institute of Technology, Delhi, Hauz Khas, New Delhi-110016, India. e-mail: Gurdevanthal92@gmail.com}\;, \bf J. M. Do \'{O} \footnote{ Department of Mathematics, Federal University of Para\'{i}ba 58051-900, Jo\~{a}o Pessoa-PB, Brazil. e-mail: jmbo@pq.cnpq.br }\;, J. Giacomoni\footnote{ 	LMAP (UMR E2S UPPA CNRS 5142) Bat. IPRA, Avenue de l'Universit\'{e}, 64013 Pau, France. 
					email: jacques.giacomoni@univ-pau.fr}  \;and  K. Sreenadh\footnote{Department of Mathematics, Indian Institute of Technology, Delhi, Hauz Khas, New Delhi-110016, India.	e-mail: sreenadh@maths.iitd.ac.in}}}
		\date{}
		
		\maketitle
		
		\begin{abstract}
\noi This article  investigates the  existence and properties of ground state solutions to the following nonlocal Hamiltonian elliptic system:
\begin{align*}
	\begin{cases}
		(-\De)^\frac12 u +V_0 u =g(v),~x\in \mb R\\
		(-\De)^\frac12 v +V_0 v =f(u),~x\in \mb R,
	\end{cases}
\end{align*}
	where $(-\De)^\frac12$ is the square root Laplacian operator, $V_0 >0$ and $f,~g$ have critical exponential growth in $\mb R$.	Using minimization technique over some generalized Nehari manifold, we show that the set $\mc S$ of ground state solutions is non empty. Moreover for $(u,v) \in \mc S$, $u,~v$ are uniformly bounded in $L^\infty(\mb R)$ and uniformly decaying at infinity. We also show that the set $\mc S$ is compact in $H^\frac12(\mb R) \times H^\frac12(\mb R)$ up to translations. Furthermore under locally lipschitz continuity  of $f$ and $g$ we obtain a suitable Poho\v{z}aev type identity for any $(u,v) \in \mc S$. We deduce the existence of semi-classical ground state solutions to the singularly perturbed system
	\begin{align*}
		\begin{cases}
			\e(-\De)^\frac12 \varphi +V(x) \varphi =g(\psi),~x\in \mb R\\
		\e	(-\De)^\frac12 \psi +V(x) \psi =f(\varphi),~x\in \mb R,
		\end{cases}
		\end{align*}
	where $\e>0$ and $V \in C(\mb R)$ satisfy the assumption $(V)$ given below (see Section \ref{I}). Finally as $\e \ra 0$, we prove the existence of minimal energy solutions which concentrate around the closest minima of the potential $V$.\\
	
			\noi \textbf{Key words:} Nonlocal Hamiltonian system, Schr\"{o}dinger system, Trudinger-Moser inequality, Critical exponential growth, Ground state solutions, Poho\v{z}aev identity, Concentration phenomena.\\
			
				\noi \textit{2020 Mathematics Subject Classification:} 35A15, 35B25, 35B33, 35J61.\\

		\end{abstract}
	\section{Introduction and Main results}\label{I}
		\noi In this article, we are concerned with existence and concentration behaviour of solutions for the following singularly perturbed coupled fractional Schr\"{o}dinger equations
	\begin{align}\label{1.1}
		\begin{cases}
		\e(-\Delta)^\frac12 \varphi +V(x)\varphi= \ds \frac{\partial K(\varphi,\psi)}{\partial \psi}\;\text{in}\;\mb R,\\
		\e(-\Delta)^\frac12 \psi+ V(x) \psi= \ds \frac{\partial K(\varphi,\psi)}{\partial \varphi}\;\text{in}\;\mb R,\\
	\end{cases}
	\end{align}
where $(-\De)^\frac12$ is the fractional Laplacian operator, defined for a measurable function $u: \mb R \ra \mb R$ by
  \begin{equation*}
 	(-\De)^\frac12 u(x)=\frac{1}{\pi}P.V.\I\mb R\frac{u(x)-u(y)}{|y|^2}dy=-\frac{1}{2\pi}P.V.\I\mb R \frac{u(x+y)+u(x-y)-2u(x)}{|y|^2}dy
 \end{equation*}
and $\e >0$. The external Schr\"{o}dinger potential satisfies the following condition:
\begin{itemize}
	\item [$(V)$] $V \in C(\mb R,~\mb R)$ and $0 < V_0:=\inf_{\mb R} V(x) < \lim_{|x| \ra \infty} V(x)=V_\infty < \infty.$
\end{itemize}
The Hamiltonian $K$ is of the form $K(\varphi,\psi)=G(\psi)-F(\varphi)$, with $ \ds F(t)=\I{0}^{t}f(s)ds$ and  $ \ds G(t)=\I{0}^{t}g(s)ds$ and the nonlinearities $f,~g$ satisfy the following assumptions:
\begin{itemize}
	\item [$(H_1)$] $f$ and $g$ are continuous functions.
	\item [$(H_2)$] $f(t) =o(t^2)$ and $g(t)=o(t^2)$ as $t \ra 0$.
	\item [$(H_3)$] There exists $\mu >2$ such that for any $t \in \mb R\setminus \{0\}$
	\begin{align*}
		0\leq\mu F(t)\leq tf(t)~\text{and}~0\leq \mu G(t)\leq g(t)t,
	\end{align*}
where $F(t) =\ds \I0^t f(s)ds$ and  $G(t) =\ds \I0^t g(s)ds$.
	\item [$(H_4)$] There exists $M>0$ such that for any $t \in \mb R \setminus \{0\}$
	\begin{align*}
		0<F(t)\leq M |f(t)|~\text{and}~0<G(t)\leq M|g(t)|.
	\end{align*}
\item  [$(H_5)$]  $f(t)/|t|$ and $g(t)/|t|$ are strictly increasing in $\mb R$.

\end{itemize}
We know from classical Sobolev embedding that $H^\frac12(\mb R) \subseteq L^p(\mb R)$ for all $ p \in [2,\infty)$, but $H^\frac12(\mb R) \not \subseteq L^\infty(\mb R)$. The following result of Ozawa \cite{TO}, later extended by Iula et al. \cite{IMM} (see also \cite{GiMiSe}):
	\begin{Lemma}\label{TME}
	For any $\ba>0$ and $u \in H^\frac12(\mb R)$, we have
	\begin{equation}\label{eTME}
		\I\mb R (\exp(\ba u^2-1))dx <\infty
	\end{equation}
\end{Lemma}
shows that $H^\frac12(\mb R)$ is also embedded in a Orlicz space of exponential growth type.
	Moreover the optimal constant in the embedding is given by
	\begin{equation*}
		\sup_{\substack {u \in H^\frac12(\mb R)\\ \|u\|_\frac12 \leq 1}} \I\mb R (\exp(\ba u^2-1))dx \begin{cases}
			<\infty~&\text{if}~\ba\leq \pi,\\
			=\infty~&\text{if}~\ba> \pi.
		\end{cases}
	\end{equation*}
 From Lemma \ref{TME}, it is not difficult to show that for any $p \geq 1$ and for any $K \subset \subset \mb R$, the map $H^\frac12(\mb R) \ni u \mapsto e^{u^2} \in L^p(K)$ is continuous with respect to the strong topology but not relatively compact with respect to the weak topology.
According to Lemma \ref{TME}, we give the definition of critical   exponential growth:
\begin{itemize}
	\item [$(H_6)$] (\textbf{critical exponential growth}) 	There exists $\ba_0 >0$ such that
	\begin{equation*}
		\lim\limits_{t \ra +\infty} \frac{f(t)}{e^{\ba t^2}}=\begin{cases}
			0~&\text{if}~\ba>\ba_0,\\
			+\infty~&\text{if}~ \ba <\ba_0
		\end{cases}
	~\text{and} ~\lim\limits_{t \ra +\infty} \frac{g(t)}{e^{\ba t^2}}=\begin{cases}
		0~&\text{if}~\ba>\ba_0,\\
		+\infty~&\text{if}~ \ba <\ba_0.
	\end{cases}
	\end{equation*}
\end{itemize}
Due to the critical behaviour in  $(H_6)$ that induces the lack of weak compactness in the Trudinger-Moser embedding, the minimizing sequences do not satisfy in general  the Palais smale type condition. To overcome this difficulty, we assume the following growth condition for getting the compactness of suitable associated Palais Smale sequences:
\begin{itemize}
	\item [$(H_7)$] There exists $\kappa_0$ and $r_1$ positive such that
		\begin{align*}
		\liminf\limits_{|t|\ra \infty}f(t)te^{-\ba_0t^2}\geq \kappa_0 >&\max\left\{\frac{8 e^\frac12}{\ba_0}V_0,\frac{ \pi}{\ba_0 r_1} \right\}\\
	&\text{and}~	\liminf\limits_{|t|\ra \infty}g(t)te^{-\ba_0t^2}\geq \kappa_0 >\max\left\{\frac{8 e^\frac12}{\ba_0}V_0,\frac{  \pi}{\ba_0 r_1} \right\}.
		\end{align*}
\end{itemize}
We also emphasize that the associated energy functional for Hamitonian elliptic systems is strongly indefinite, that is, its quadratic part is respectively coercive and anti-coercive in infinite dimensional subspaces of the energy space. This poses additional difficulties to get a priori estimates in the study of such systems.\\
In this paper, we will focus on the existence of the minimal energy solutions, the so-called ground states. These solutions play a fundamental role both from the theoretical point of view as well as from the physical meaning (see for instance \cite{BF}).\\
The study of elliptic equations involving fractional operators has gained huge interest in light of its real world applications. Indeed, these operators occur naturally in several physical phenomena. Among the others we can quote models in population dynamics, models of American options in finance, game theory, image processing, see for instance \cite{C, DPV} and references therein. For models involving critical exponential growth nonlinearities and fractional diffusion operators, we refer specifically to \cite{BCM, GiMiSe, IS}.\\
The existence of solutions for Hamiltonian elliptic systems with standard Laplacian operator, $i.e.$ the local case, is widely studied in higher dimensions involving Sobolev critical growth both in bounded and unbounded domains of $\mb R^n$. For the bounded domain case, we refer to \cite{BR,CFM,FF,HV} while for the case of whole space $\mb R^n$ we address the references \cite{LY,S,SS}.\\
For $n=2$, the semilinear elliptic systems with exponential nonlinearity in bounded domains have been studied by D. G. de Figueiredo et al. \cite{FOR} in the critical growth range. Later de Figueiredo et al. \cite{FOZ} considered the case of whole space $\mb R^2$ and showed the existence of ground states using Nehari manifold technique and assuming  $f(t)/|t|$ and $g(t)/|t|$ are strictly increasing. D. Cassani and J. Zhang \cite{CZ} extended this study and prove some a priori estimates of the set of ground states solutions and further investigate the existence and concentration profile of semi-classical ground state solutions to the following singularly perturbed system:
\begin{align*}
	\begin{cases}
		-\e^2 \De \varphi +V(x) \varphi =g(\psi)~\text{in}~\mb R^2,\\
		-\e^2 \De \psi +V(x) \psi =f(\varphi)~\text{in}~\mb R^2,
	\end{cases}
\end{align*}
under suitable assumptions on the potential $V$ and the nonlinearities $f$ and $g$. In addition, we would like to mention the works \cite{LL,QTZ} where the authors studied the Hamiltonian elliptic systems without the Ambrosetti-Rabinowitz condition.\\
Regarding the system of equations with fractional Laplacian operators and with polynomial subcritical and critical nonlinearities, we refer to \cite{CD, HSZ, ZDS} and the references therein. The case of fractional elliptic systems over $\mb R$ with critical exponential growth nonlinearities was less explored. We refer to \cite{CMM, AM, DGM, GMS} for the study of such problems. In \cite{DGM}, do \'{O} et al. considered the following nonautonomous fractional Hamiltonian system with critical exponential growth
\begin{align*}
	\begin{cases}
		(-\De)^\frac12 u+u =Q(x)g(v)~\text{in}~\mb R,\\
		(-\De)^\frac12 v+v =P(x)f(u)~\text{in}~\mb R,
	\end{cases}
\end{align*} 
where the nonnegative weights $P(x)$ and $Q(x)$ vanish at infinity. By using suitable variational method combined with a general version of the linking theorem, the authors obtained the existence of at least one positive solution. In \cite{ADM, Am} results about concentration phenomena for fractional elliptic scalar equations and systems with polynomial critical growth respectively  are obtained {In \cite{ADM}, the authors make use of Caffarelli-Silvestre harmonic extension method but this method is difficult to implement for the case of systems.  }\\
Motivated by the above literature, we consider in the present paper a class of Hamiltonian elliptic systems posed in the whole space $\mb R$ and involving the square root Laplacian  and critical  exponential nonlinearities. For this class of problems, the existence and the behaviour of weak solutions  have not been investigated in former contributions. 
 In order to study the singularly perturbed system \eqref{1.1}, we first investigate the existence and a priori estimates of the following so-called limit problem:
\begin{align}\label{s1.4}
	\begin{cases}	
		(-\De)^\frac12 u +V_0u =g(v),\\
		(-\De)^\frac12v +V_0v =f(u).
	\end{cases}
\end{align}
Our approach is based on Nehari Manifold technique in the spirit of Pankov \cite{P} to get the existence of the ground-states of system \eqref{s1.4}. We then obtain a priori estimates of these ground-states and later employ those results to analyse the concentration profile of semiclassical solutions.\\
In the following theorem, we state the main existence result about ground state solutions to the system \eqref{s1.4}:
\begin{Theorem}\label{t1.1}
	Suppose that $f$, $g$ satisfy $(H_1)-(H_7)$. Then \eqref{s1.4} has a ground state solution.
\end{Theorem}
\begin{Remark}
\begin{enumerate}
	\item	Note that to get positive ground state solutions, we can extend $f$ and $g$ with $0$ on $\mb R^-$. By taking $(0,v^-)$ and $(u^-,0)$ as test functions in \eqref{edj} one can conclude $v^- \equiv 0$ and $u^- \equiv 0$ respectively. Proving regularity results of $(u,v)$ one can further apply the strong maximum principle for classical solutions.
	\item {A similar result to the previous theorem was proved in \cite{AM}. But here the coupling is assumed only in the linear perturbation and nonlinearities considered are $C^1$. This makes the associated Nehari manifold fold of class $C^1$. So the usual Nehari manifold approach can be applied to find the critical points of associated energy functional. In our case, since $f,~g$ are merely continuous, the Nehari manifold is not of class $C^1$. So we cannot use the usual Nehari manifold approach. To overcome this, we make use of the ideas developed in \cite{SW1} to get the existence of a Palais-Smale sequence.}
	\end{enumerate}
\end{Remark}
Denote by $\mc S$ the set of ground state solutions to the system \eqref{s1.4}. We further investigate the regularity and qualitative properties of the ground state solutions to \eqref{s1.4}. 
Precisely, we prove the following results:
\begin{Theorem}\label{t1.3}
	Suppose that $f$ and $g$ satisfy $(H_1),~(H_2),~(H_3)$ and $(H_6)$. Then the following assertions hold:
	\begin{enumerate}[label =(\roman*)]
		\item $(u,v) \in \mc S \implies u,~v \in L^\infty(\mb R) \cap C_{loc}^{0,\gamma}(\mb R)$ for all $\gamma \in (0,1)$;
		\item Let $x_w \in \mb R$ be the maximum point of $|u(x)|+|v(x)|$, then the set 
		\begin{align*}
			\{(u(\cdot + x_w),~v(\cdot+x_w))| ~(u,v)\in \mc S\}
		\end{align*}
		is compact in $H^\frac12(\mb R)\times H^\frac12(\mb R)$;\label{2}
		\item $0< \inf\{\|u\|_\infty,~\|v\|_\infty:~(u,v)\in \mc S\} \leq \sup \{\|u\|_\infty,~\|v\|_\infty:~(u,v)\in \mc S\}<\infty$; 
		\item $u(x+x_w) \ra 0$ and $v(x+x_w) \ra 0$, as $|x| \ra \infty$ uniformly for any $w =(u,v)\in \mc S$, where $x_w$ is given in \ref{2}.
	\end{enumerate}
\end{Theorem}
\begin{Remark}
\begin{enumerate}
	\item	{The nonlocal nature of the operator introduced new technical issues to obtain the regularity results. To overcome these issues we make use of the Calderon-Zymund types estimates for nonlocal PDE obtained in \cite{MSY} (see Proposition \ref{p4.1}).}
	\item {We also want to point out that there is a gap in the compactness result proved for the case of systems involving local operators (see \cite[equation (3.26)]{CZ}). Here, we introduced new arguments to complete this gap (see proof of Proposition \ref{p4.2}).  }
	\end{enumerate}
\end{Remark}
Next we issue the validity of Poho\v{z}aev type identity for systems of the kind \eqref{s1.4}, 
and prove the following completely new result which is of independent interest:
\begin{Theorem}\label{pt5.1}
	Let $w=(u,v)$ be a solution to \eqref{s1.4} and assume $f, g$ are locally lipschitz continuous. Then the following identity holds
	\begin{equation}\label{e5.1}
		\I\mb R (F(u)+G(v)-V_0uv)dx =0.
	\end{equation}
\end{Theorem}
\begin{Remark}
	{The proof Poho\v{z}aev identity poses serious difficulties due to the lack of general variational inequalities in the fractional operator case. So we introduced a completely new approach (see Section \ref{s5}) which is of independent interest to obtain \eqref{e5.1}}.
\end{Remark}
Finally we have the following result concerning the existence and concentration of ground states of system \eqref{1.1}:
\begin{Theorem}\label{t6.1}
	Assume condition $(V)$ and that $f$, $g$ satisfy $(H_1)-(H_7)$. Then for sufficiently small $\e>0$, \eqref{1.1} admits a ground state solution $w_\e=(\varphi_\e,\psi_\e) \in W$. Moreover, if $x_\e$ is any maximum point of $|\varphi_\e|+|\psi_\e|$, then, setting 
	\begin{equation*}
		\mc L \equiv \{x \in \mb R: V(x)=V_0\}
	\end{equation*}
	one has
	\begin{equation*}
		\lim\limits_{\e \ra 0} dist(x_\e,\mc L)=0~\text{and}~\lim\limits_{\e\ra 0}|x_\e^k-x_\e|=0,~k=1,2.
	\end{equation*}
	Furthermore, $(\varphi_\e(\e x+x_\e),~\psi_\e(\e x+x_\e))$ converges (up to a subsequence) as $\e \ra 0$ to a ground state solution of 
	\begin{align*}
		\begin{cases}	
			(-\De)^\frac12 u +V_0u =g(v),\\
			(-\De)^\frac12v +V_0v =f(u).
		\end{cases}
	\end{align*}
\end{Theorem}
 	Throughout this paper without loss of generality, we may assume $0 \in \mc L$.
\subsection{Structure of the paper} In Section \ref{FNM}, we present the required functional settings and give the mini-max characterization of the corresponding energy level using the Generalized Nehari manifold approach. In Section \ref{s3}, we show the existence of ground state solutions to the system \eqref{s1.4}. Section \ref{s4} is devoted to obtain a priori estimates, regularity and qualitative properties of solutions. In Section \ref{s5}, using some original ideas from \cite{BMS} the Poho\v{z}aev identity \eqref{e5.1} is established. Finally in Section \ref{s6} we study the singularly perturbed system \eqref{1.1}. Gathering some informations previously obtained in Sections \ref{s3} and \ref{s4} we establish the concentration behaviour of the semi-classical solutions.
\section{Functional setting and generalized Nehari manifold}\label {FNM}	
We recall that Sobolev 	space $H^\frac12(\mb R)$ is the collection
\begin{equation*}
	H^\frac12(\mb R)=\left\{u \in L^2(\mb R): \frac{u(x)-u(y)}{|x-y|} \in L^2(\mb R \times \mb R)\right\},
\end{equation*}
endowed with the inner product 
\begin{align*}
	\langle u,v\rangle_\frac12 =\frac{1}{2\pi}\I\mb R\I\mb R \frac{(u(x)-u(y))(v(x)-v(y))}{|x-y|^2}dxdy +\I\mb R V_0 uvdx,~u,v \in H^\frac12(\mb R),
\end{align*}	
and the associated norm is given by $\ds\|u\|_\frac12^2=\langle u,u\rangle_\frac12.	$
\begin{Remark}
	In view of \cite[Proposition 3.6]{DPV}, we also have
	\begin{align*}
		\| (-\De)^\frac14u\|_{L^2}^2:= \frac{1}{2\pi}\I\mb R\I\mb R \frac{|u(x)-u(y)|^2}{|x-y|^2}dxdy,~\text{for all}~u \in H^\frac12(\mb R).
	\end{align*}
Here the fractional Laplacian $(-\De)^\frac14$ of a function $u: \mb R \ra \mb R$ in the Schwartz class is defined by 
\begin{equation*}
	(-\De)^\frac14u =\mc F^{-1 }(|\xi|^\frac12(\mc Fu)),~\text{for all }~\xi \in \mb R,
\end{equation*}
where $\mc F$ denotes the Fourier transform
\begin{equation*}
	\mc F(u)(\xi)=\frac{1}{\sqrt{2\pi}}\I \mb R e^{-i\xi.x}u(x)dx.
\end{equation*}
Thus we also have 
\begin{equation*}
	\|u\|_\frac{1}{2}^2=\|(-\De)^\frac14 u\|_{L^2}^2+V_0\|u\|_{L^2}^2.
	\end{equation*}
	\end{Remark}

The natural space that contains all the solutions of problem \eqref{s1.4} is the space $W =H^\frac12(\mb R)\times H^\frac12(\mb R)$ 	with the inner product 
\begin{equation*}
	\langle w_1,w_2 \rangle_W := \langle u_1,u_2\rangle_\frac12 +\langle v_1,v_2\rangle_\frac12 ,~w_k =(u_k,v_k)\in W,~ k=1,2,
\end{equation*}
and the associated norm $\|	w\|^2_W=\|(u,v)\|^2_W=\|u\|_\frac12^2 +\|v\|_\frac12^2.$
We introduce  the following subspaces
\begin{align*}
	W^+:= \{(u,u)|~ u \in H^\frac12(\mb R)\}~\text{and}~W^-:= \{(u,-u)|~ u \in H^\frac12(\mb R)\}.
\end{align*}	
	Then for each $w =(u,v) \in W$, we have the straightforward decomposition:
	\begin{equation}\label{2.3}
	w=w^+ +w^- =\left( (u+v)/2,(u+v)/2\right) +\left((u-v)/2,(v-u)/2\right).
\end{equation}
Thus it is clear that $W =W^+ \oplus W^-$.\\	
We define the associated energy functional to problem \eqref{s1.4} by
	\begin{equation*}
		\mc J (w)=\mc J(u,v):= \langle u,v\rangle_\frac12 -\Phi (w),~w \in W,
	\end{equation*}
where $\ds \Phi (w)= \I\mb R(F(u)+G(v))dx$.	
 By the hypothesis on $f$ and $g$, we see that 
\begin{align*}
	\Phi(0)=0,~\langle \Phi'(w),w\rangle >2 \Phi(w)>0,~\text{for all}~w \in W\setminus\{0\}.
	\end{align*}
Also using Lemma \ref{TME} and from standard arguments, we can see that $\mc J$ is well-defined and of class $C^1$ with 
	\begin{align}\label{edj}
		\langle \mc J'(u,v),(\varphi,\psi)\rangle=\langle u,\psi\rangle_\frac12+\langle v,\varphi\rangle_\frac12-\I\mb R (f(u)\varphi+g(v)\psi)dx, ~\varphi,~\psi \in H^\frac12(\mb R).
	\end{align}
\begin{Definition}
		We call $w =(u,v) $ a weak solution of \eqref{s1.4} if $w \in W$ and is a critical point of the associated energy functional $\mc J$.
\end{Definition}
Some remarks are in order:
\begin{Remark}\label{r2.4}
	\begin{enumerate}
		\item Take $w=(u,v) \in W$. Then using the decomposition in \eqref{2.3}, we easily see that
		\begin{equation*}
			\mc J(w)=\frac12 \|w^+\|_W^2 -\frac12\|w^-\|_W^2-\Phi(w),
		\end{equation*}
	which shows that $\mc J$ is a strongly indefinite functional.
	\item If $w=(u,v) \in W\setminus\{0\}$ and $\mc J'(w)=0$, then by $(H_3)$
	\begin{align*}
		\mc J(w) &=\mc J(w)-\frac12 \langle \mc J'(w),w\rangle \\
		&= \I\mb R \left(\frac12 f(u)u-F(u)\right)dx+\I\mb R \left(\frac12 g(u)u-G(u)\right)dx>0.
		\end{align*}
	On the other hand, if $w=(u,-u)\in W^-$, we have by $(H_3)$ that 
	\begin{align*}
		\mc J(w)=-\|u\|_\frac12^2 -\Phi(w)\leq 0.
	\end{align*}
Consequently, if $w \in W$ is a nontrivial critical point of $\mc J$, then $w \in W\setminus W^-$.
	\end{enumerate}
\end{Remark}
The ground state solutions of \eqref{s1.4} are obtained by minimizing the energy functional $\mc J$ over the generalized Nehari manifold
\begin{equation*}
	\mc N= \left \{ w \in W\setminus W^- : \langle \mc  J'(w),w \rangle=0,~\langle \mc  J'(w),\varphi \rangle=0,~\text{for all }~\varphi \in W^- \right\}
\end{equation*}
 whose type was initially introduced by \cite[Pankov]{P}. In view of Remark \ref{r2.4}, we see that any nontrivial critical point of $\mc J$ {in fact} lies in $\mc N$. The ground state solutions will be obtained as nontrivial critical points of $\mc J$ in $\mc N$. To obtain such points, for any $w \in W \setminus W^-$, we set 
\begin{equation*}
\widetilde	{W}(w):= W^- \oplus \mb R^+ w,
\end{equation*}
where $\mb R^+ := \{t \in \mb R:~ t \geq 0\}$. Then observe that
\begin{equation*}
\widetilde	{W}(w) =W^- \oplus \mb R^+ w^+ =\widetilde{W}(w^+).
\end{equation*}
We note that the various results proved in \cite[Section 2]{FOZ} for the local system in dimension two can be straightforwardly extended for the nonlocal system \eqref{s1.4} also. In the following we merely state these results omitting their proofs:
\begin{Lemma}\label{l3.1}
	Assume that $f$, $g$ satisfy $(H_1)-(H_4)$. Then for each $w \in W\setminus W^-$, there exists $\rho =\rho(w)>0$ such that 
	\begin{equation*}
		\mc J(z)\leq 0~\text{if}~z \in \widetilde{W}(w)\setminus B_\rho(0),
	\end{equation*}
where $B_\rho(0):=\{\ell \in W : \|\ell\|_W <\rho\}.$
	\end{Lemma}

The ground state solutions will be obtained by maximizing $\mc J$ over the sets $\widetilde{W}(w)$. For that matter, we start with the following proposition.
\begin{Proposition}\label{p2.6}
With the assumptions of Theorem \ref{t1.1}, we have
\begin{enumerate}
	\item [(1)] for any $w \in \mc N$, $\mc J\arrowvert_{\widetilde{W}(w)}$ admits a unique maximum point which is precisely at $w$;
	\item [(2)] for any $w \in W\setminus W^-$, the set $\widetilde{W}(w)$ intersects $\mc N$ at exactly one point $\widetilde{m}(w)$, which is the unique global maximum point of $\mc J\arrowvert _{ \widetilde{W}(w)}$.
\end{enumerate}	
\end{Proposition}
\begin{Corollary}\label{c3.3} We have,
\begin{equation*}
\inf_{\eta \in \mc N}\mc J(\eta)=\inf_{w\in W\setminus W^-} \max_{z\in \widetilde{W}(w)}\mc J(z).
\end{equation*}
\end{Corollary}
\begin{Proposition}\label{p3.4}
	Let 
	\begin{equation*}
		m^\ast := \inf_{w \in \mc N} \mc J(w),
	\end{equation*}
then $m^\ast >0$.
\end{Proposition}
%
\begin{Proposition}\label{p3.5}
There exists $\de >0$ such that $\|w^+\|_W \geq \de$ for all $w \in \mc N$. In particular,
\begin{align*}
	\| \widetilde{m}(w)^+\|_W \geq \de~\text{for all}~w \in W\setminus W^-.
	\end{align*}
Moreover, for each compact subset $\mc K \subset W\setminus W^-$, there exists a constant $C_{\mc K}>0$ such that
\begin{equation*}
	\|\widetilde{m}(w)\|_W \leq C_{\mc K},~\text{for all}~w \in \mc K.
\end{equation*}
\end{Proposition}
For each $z\in W\setminus W^-$, let 
\begin{align*}
	W(z)=W^- \oplus \mb R z =W^- \oplus \mb R^+z^+.
\end{align*}
Let 
\begin{equation*}
	S^+ := S\cap W^+ =\{w \in W^+ : \|w\|_W =1\}.
\end{equation*}
Obviously, $S^+$ is a $C^1-$ submanifold of $W^+$ and the tangent space of $S^+$ at $w \in S^+$ is 
\begin{equation*}
	T_w(S^+)=\{z \in W^+:(z,w)=0\}.
\end{equation*}
Then for any $w \in S^+$,
\begin{equation*}
	W =T_w(S^+)\oplus \mb Rw \oplus W^- =T_w(S^+)\oplus W(z).
\end{equation*}
Noting that $f$ and $g$ are merely continuous, the manifold $\mc N$ is not in general of class $C^1$.
Thus, the Nehari manifold technique cannot be used in a straghtforward way to find the critical points
of $\mc J$ in $\mc N$ . For this issue, following the approach used in \cite{SW1}, we consider another
related functional $\mc F : S^+ \ra  \mb R$ where for each $w \in S^+ \subset W\setminus W^-$, $\mc F(w)$ is defined as
$\mc F(w) := \mc J(m(w))$ with
\begin{equation*}
	m :=\widetilde{m}\arrowvert_{S^+}: S^+ \mapsto \mc N.
\end{equation*}
Now we have the
following propositions, which are consequences of \cite[Proposition 31]{SW1} and \cite[Corollary 33]{SW1} respectively:
\begin{Proposition}\label{p2.10}
	The mapping $\widetilde{m}:W\setminus W^- \ra \widetilde{W}(w)\cap \mc N$ is continuous and the mapping $m: S^+ \mapsto \mc N$ is a homeomorphism.
\end{Proposition}
\begin{Proposition}\label{p2.11} We have
	\begin{itemize}
		\item [(a)]$\mc F \in C^1(S^+,\mb R)$ and
		\begin{equation*}
			\langle \mc F'(w),z\rangle=\|m(w)^+\|_W\langle \mc J'(m(w)),z\rangle~\text{for all}~z \in T_w(S^+).
		\end{equation*}
	\item [(b)] If $\{z_n\}\subset S^+$ is a Palais-Smale sequence for $\mc F$, then $\{m(z_n)\}\subset \mc N$ is a Palais-Smale sequence for $\mc J$. Namely, if $\mc F(z_n)\ra d$ for some $d>0$ and $\|\mc F'(z_n)\|_\ast \ra 0$ as $n \ra \infty$, then $\mc J(m(z_n)) \ra d$ and $\| \mc J'(m(z_n))\| \ra 0$ as $n \ra \infty$, where
	\begin{equation*}
		\|\mc F'(z_n)\|_\ast =\sup_{\substack{\phi \in T_{z_n}(S^+)\\ \|\phi\|_W=1}}\langle \mc F'(z_n),\phi\rangle
	~\text{and}~
	\|\mc J'(m(z_n))\| =\sup_{\substack{\phi \in W\\ \|\phi\|_W=1}}\langle {\mc J'(m(z_n))},\phi\rangle.
\end{equation*}
\item [(c)] $z \in S^+$ is a critical point of $\mc F$ if and only if $m(z) \in \mc N$ is a critical point of $\mc J$.
\item [(d)] $\inf_{S^+}\mc F =\inf_{\mc N} \mc J$.
	\end{itemize}
\end{Proposition}

\section{Proof of Theorem \ref{t1.1}}\label{s3}
By Proposition \ref{p3.4} and Proposition \ref{p2.11}, we have
$ \ds \inf_{S^+}\mc F =\inf_{\mc N} \mc J =m^\ast >0$. Since $S^+$
is a regular $C^1$-submanifold of $W^+$ and from the Ekeland variational principle,  there exists $\{z_n\} \subset S^+$ such that 
\begin{align*}
	\mc F(z_n)\ra m^\ast >0~\text{and}~\|\mc F'(z_n)\|_\ast \ra 0~\text{as}~n \ra \infty.
\end{align*}
Let $w_n =m(z_n) \in \mc N$. Then applying Proposition \ref{p2.11}, we have
\begin{align*}
	\mc J(w_n)\ra m^\ast>0~\text{and}~\| \mc J'(w_n)\| \ra 0~\text{as}~n \ra \infty.
\end{align*}
In the next proposition we prove that the sequence $\{w_n\}_{n\in\mathbb{N}}$ is bounded in $W$. Precisely, we have
\begin{Proposition}\label{p3.1}
	There exists $C>0$ independent of $m^*$ such that for all $n$
	\begin{enumerate}
		\item [(1)] $\|w_n\|_W=\|(u_n,v_n)\|_W \leq C (m^*+1)$;
		\item [(2)] $\I \mb R f(u_n)u_n dx\leq C(m^*+1)$ and $\I \mb R g(v_n)v_n dx\leq C(m^*+1)$;
		\item [(3)] $\I \mb R F(u_n)dx \leq C(m^*+1)$ and $\I \mb R G(v_n)dx \leq C(m^*+1)$.
	\end{enumerate}
\end{Proposition}

\begin{proof}
	The proof follows similarly as the proof of \cite[Proposition 3.1]{FOZ}.
 \QED
	\end{proof}
Since $\{w_n\}_{n\in\mathbb{N}}$ is bounded, $w_n \rightharpoonup w_0=(u_0,v_0)$ weakly in $W$. Now, we have the following proposition:
\begin{Proposition}
	The limit $w_0$ of the sequence $\{w_n\}$ is critical point of $\mc J$.
\end{Proposition}
\begin{proof}
	The proof follows easily from Proposition \ref{p3.1} together with the compact embedding of $H^\frac12(\mb R) \hookrightarrow L^r_{\text{loc}}(\mb R)$ and Vitali's convergence theorem.
 \QED
	\end{proof}
The final step is to prove $u_0 \ne 0$ and $v_0 \ne 0$. To this aim, the following key result plays a crucial role to get strong convergence of the minimizing sequence $\{w_n\}_{n\in\mathbb{N}}$.
\begin{Proposition}\label{pp3.4}
	The level $m^\ast $ of the minimal energy defined in Proposition \ref{p3.4} satisfies
	\begin{equation}\label{e3.10}
	0<m^\ast<\frac{\pi}{\ba_0}.
	\end{equation}
\end{Proposition}
\begin{proof}
	For this purpose, for $r_1$ given by $(H_7)$, we consider the following sequence of Moser's functions
	\begin{align}\label{ms}
		\omega_n=\begin{cases}
			(\log n)^\frac12,~&\mbox{if }0\leq |x|\leq \frac{r_1}{n},\\
			\frac{\log(r_1/|x|)}{(\log n)^\frac12}, ~&\mbox{ if }\frac{r_1}{n}\leq|x|\leq r_1,\\
			0,~&\mbox{ if }|x| \geq r_1.
		\end{cases}
	\end{align}
Then from \cite[Lemma 3.4]{GiMiSe}, we have the following estimates
\begin{equation*}
\| (-\De)^\frac14 \omega_n\|_{L^2}^2 =\pi \text{ and}~\lim\limits_{n\ra \infty} \|\omega_n\|_{L^2}^2=O((\log n)^{-1}).
\end{equation*}
{Setting} $\ds \hat{\omega}_n=\frac{\omega_n}{\|\omega_n\|_{\frac{1}{2}}}$, for $x \in B_\frac{r_1}{n}(0)$ and $n$ sufficiently large, we then have
\begin{equation}\label{e3.11}
	\hat{\omega}_n^2 \geq \frac{\log n}{\pi+O((\log{n})^{-1})}.
	\end{equation}
Now thanks to Corollary \ref{c3.3} and Proposition \ref{p3.4}, to prove \eqref{e3.10}, it is sufficient to show that there exists $n \in \mb N$ such that
\begin{equation*}
	\sup_{w \in \widetilde{W}(\hat{\omega}_n,\hat{\omega}_n)}\mc J(w) <\frac{\pi}{\ba_0}.
\end{equation*}
We follow the proof of \cite[Proposition 3.2]{CZ} in this regard.
Suppose on the contrary that this is not true. Then for each $n \in \mb N$, we have
\begin{equation*}
	\sup_{w \in \widetilde{W}(\hat{\omega}_n,\hat{\omega}_n)}\mc J(w) \geq \frac{\pi}{\ba_0}.
	\end{equation*}
Using Proposition \ref{p2.6}, we get
\begin{equation}\label{e3.13}
	\mc J(\widetilde{m}(\hat{\omega}_n,\hat{\omega}_n)) \geq \frac{\pi}{\ba_0}~\text{for all}~n,
\end{equation}
where $\widetilde{m}(\hat{\omega}_n,\hat{\omega}_n)\in \mc N \cap \widetilde{W}(\hat{\omega}_n,\hat{\omega}_n)$. Now writing
\begin{equation*}
	\widetilde{m}(\hat{\omega}_n,\hat{\omega}_n)=\tau_n(\hat{\omega}_n,\hat{\omega}_n)+(u_n,-u_n),~\tau_n \in \mb R^+,~u_n \in H^\frac12(\mb R),
\end{equation*}
using \eqref{e3.13} and the fact that  $ \widetilde{m}(\hat{\omega}_n,\hat{\omega}_n) \in \mc N$, we have 
\begin{equation*}
	\tau_n^2-\|u_n\|_\frac12^2 -\I\mb R [F(\tau_n\hat{\omega}_n+u_n)+G(\tau_n\hat{\omega}_n-u_n)]dx \geq \frac{\pi}{\ba_0}
\end{equation*}
and
\begin{equation}\label{e3.15}
	\tau_n^2-\|u_n\|_\frac12^2 =\frac12\I\mb R [f(\tau_n\hat{\omega}_n+u_n)(\tau_n\hat{\omega}_n+u_n)+g(\tau_n\hat{\omega}_n-u_n)(\tau_n\hat{\omega}_n-u_n)]dx.
\end{equation}
By $(H_7)$ given $\e >0$, we can find $R_\e >0$ such that 
\begin{equation}\label{e3.16}
	tf(t),~tg(t)\geq (\gamma_0-\e)e^{\ba_0t^2}~\text{for all}~t \geq R_\e.
\end{equation}
Also by noting that 
$	\tau_n^2 \hat{\omega}_n^2 \geq\frac{\log{n}}{\beta_0}(1-O((\log{n})^{-1}))\text{ on}~B_\frac{r_1}{n}(0),$
we see that for $n$ large enough, $\max\{\tau_n\hat{\omega}_n+u_n,\tau_n\hat{\omega}_n-u_n\} \geq R_\e$ for all $\ds x \in B_\frac{r_1}{n}(0)$ and so using \eqref{e3.11}, \eqref{e3.15} and \eqref{e3.16}, we have
\begin{align}\label{e3.17}
\displaystyle	\tau_n^2 \geq \displaystyle\frac{(\gamma_0-\e)}{2}\I{B_\frac{r_1}{n}(0)}e^{\ba_0\tau_n^2\hat{\omega}_n^2} dx\geq \displaystyle (\gamma_0-\e)r_1e^{\frac{\ba_0}{\pi+ O((\log{n})^{-1})}\tau_n^2\log n-\log n}.
\end{align}
From \eqref{e3.17}, we conclude that $\{\tau_n\}$ is bounded and also $\ds \tau_n^2 \ra \frac{\pi}{\ba_0}$ as $n \ra \infty$. Now taking into account that $\hat{\omega}_n \ra 0$ as $n \ra \infty$ a.e. in $\mb R$ and by the Lebesgue dominated convergence theorem, we get as $n \ra \infty$
\begin{equation}\label{e3.18}
	\int\limits_{\{x \in B_{r_1}(0):\tau_n\hat{\omega}_n<
		R_\e\}}e^{\ba_0(\tau_n\hat{\omega}_n)^2}dx\ra 2r_1
\end{equation}
and 
\begin{equation}\label{e3.19}
	\int\limits_{\{x \in B_{r_1}(0):\tau_n\hat{\omega}_n<R_\e\}}\min\{f(\tau_n\hat{w}_n)\tau_n\hat{\omega}_n,~g(\tau_n\hat{\omega}_n)\tau_n\hat{\omega}_n\}dx\ra 0.
\end{equation}
Again from \eqref{e3.15} together with \eqref{e3.18} and \eqref{e3.19} , we have
\begin{align*}
	\tau_n^2 \geq& \frac{1}{2}\I {B_{r_1}(0)}[f(\tau_n\hat{\omega}_n+u_n)(\tau_n\hat{\omega}_n+u_n)+g(\tau_n\hat{\omega}_n-u_n)(\tau_n\hat{\omega}_n-u_n)]dx\\
	\geq& (\gamma_0-\e)\I{B_{r_1}(0)}e^{\ba_0(\tau_n\hat{\omega}_n)^2}dx-(\gamma_0-\e)\int\limits_{\{x \in B_{r_1}(0):\tau_n\hat{\omega}_n<R_\e\}}e^{\ba_0(\tau_n\hat{\omega}_n)^2}dx\\
	&+\frac{1}{2}\int\limits_{\{x \in B_{r_1}(0):\tau_n\hat{\omega}_n<R_\e\}}\min\{f(\tau_n\hat{\omega}_n)\tau_n\hat{\omega}_n,~g(\tau_n\hat{\omega}_n)\tau_n\hat{\omega}_n\}dx\\
	=&(\gamma_0 -\e)\left(\I{B_{r_1}(0)}e^{\ba_0(\tau_n\hat{\omega}_n)^2}dx-2r_1\right)+o_n(1).
\end{align*}
Now, we estimate the term $\ds \I{B_{r_1}(0)}e^{\ba_0(\tau_n\hat{\omega}_n)^2}dx$.  Using \eqref{e3.11}, we have that
\begin{equation*}
\I{B_\frac{r_1}{n}(0)}e^{\ba_0(\tau_n\hat{\omega}_n)^2}dx	\geq 2r_1e^{\frac{\ba_0}{\pi+ O((\log{n})^{-1})}\tau_n^2\log n-\log n}.
\end{equation*}
Noting that $\ds \tau_n^2 \geq \frac{\pi}{\ba_0}$, $\ds \tau_n^2 \ra \frac{\pi}{\ba_0}$, we have
\begin{equation*}
	\liminf\limits_{n \ra \infty}\I{B_\frac{r_1}{n}(0)}e^{\ba_0(\tau_n\hat{\omega}_n)^2}dx	\geq 2r_1.
\end{equation*}
Using the change of variable $x =r_1e^{(-\|\omega_n\|_\frac12\sqrt{\frac{2\log n}{\pi}}y)}$, we get
\begin{align*}
	\I\frac{r_1}{n}^{r_1} e^{\pi \hat{\omega}_n^2}dx \geq r_1\left(1-e^{-\log n}\right).
\end{align*}
Thus
\begin{equation*}
	\liminf \I{B_{r_1}(0)}e^{\ba_0(\tau_n\hat{\omega}_n)^2}dx \geq 3r_1,
\end{equation*}
which implies that 
\begin{align*}
	\frac{\pi}{\ba_0}=\lim\limits_{n \ra \infty}\tau_n^2 \geq (\gamma_0-\e)r_1.
\end{align*}
Since $\e>0$ was arbitrary, we have
\begin{equation*}
	\gamma_0 \leq \frac{\pi}{\ba_0 r_1 },
\end{equation*}
which is a contradiction to $(H_7)$. This completes the proof. \QED
	\end{proof}
In the sequel, we use the following lemma in the setting of fractional Sobolev spaces (see \cite[lemma 2.6]{AM})  which is a consequence of Concentration Compactness Principle (\cite{Lions}):
\begin{Lemma}\label{l3.5}
	Assume that $\{\ell_n\}$ is a bounded sequence in $H^\frac12(\mb R)$ satisfying
	\begin{equation*}
		\lim\limits_{n \ra \infty} \sup_{y \in \mb R} \I{B_R(y)}\ell_n^2dx=0,
	\end{equation*}
	for some $ R>0$. Then $\ell_n \ra 0$ strongly in $L^p(\mb R)$, for $2 <p<\infty$.
\end{Lemma}
Concerning the minimizing sequence $\{w_n\}$ with $w_n=(u_n, v_n)$, we have the following lemma:
\begin{Lemma}\label{l3.6}
	There exists $R>0$ such that 
	\begin{equation*}
		\lim\limits_{n \ra \infty} \sup_{y \in \mb R} \I{B_R(y)}(u_n^2 +v_n^2)dx >0.
	\end{equation*}
\end{Lemma}
\begin{proof}
	The proof follows from the proof of \cite[Proposition 2.2]{CZ} with the help of Lemma \ref{l3.5} and slight modifications. 
 \QED
	\end{proof}
Our final step is to prove that \eqref{s1.4} has a nontrivial ground state solution. For that matter, we use the concentration compactness lemma \cite[Lemma I.1]{Lions} to get that $w_n$ satisfies only one of the following conditions:\\
Vanishing:
\begin{equation*}
	\lim\limits_{n \ra \infty} \sup_{y \in \mb R} \I{B_R(y)}(u_n^2 +v_n^2)dx=0~\text{for all}~R>0.
\end{equation*}
Nonvanishing: there exists $\nu >0$, $R'>0$ and $\{y_n\}\subset \mb R$ such that
\begin{equation*}
	\lim\limits_{n \ra \infty} \I{B_{R'}(y_n)}(u_n^2+v_n^2)dx \geq \nu.
\end{equation*}
From Lemma \ref{l3.6} one has that vanishing does not occur. So thanks to this compactness property we can complete the proof of Theorem \ref{t1.1}.\\
\textbf{End of Theorem \ref{t1.1}}: 
Let $\widetilde{u}_n(\cdot):=u_n(\cdot +y_n)$ and $\widetilde{v}_n(\cdot):= v_n (\cdot +y_n)$. Then
\begin{equation}\label{e3.33}
	\lim\limits_{n \ra \infty} \I{B_{R_0}(0)}(\tilde{u}_n^2+\tilde{v}_n^2)dx \geq \nu,
	\end{equation}
and $\mc J(\widetilde{w}_n) \ra m^\ast>0$ and $ \mc J'(\widetilde{w}_n) \ra 0$ as $n \ra \infty$, where $\widetilde{w}_n =(\widetilde{u}_n,\widetilde{v}_n)$. Obviously, $\{\widetilde{w}_n\}$ is bounded in $W$. Up to a subsequence, we assume that $\widetilde{w}_n \rightharpoonup \widetilde{w}$ weakly in $W$ for some $\widetilde{w} =(\widetilde{u},\widetilde{v}) \in W$. Similarly, as above, $\mc J'(\widetilde{w})=0$. By \eqref{e3.33}, $\widetilde{u}\neq 0$ and $\widetilde{v}\neq 0$. By $(H_3)$ and Fatou's lemma, we have
\begin{align*}
	{m^*} +o_n(1) =&\mc J(\widetilde{w}_n)-\frac12 \ld \mc J'(\widetilde{w}_n),\widetilde{w}_n\rd\\
	=& \I\mb R\left[\frac12 f(\widetilde{u}_n)\widetilde{u}_n-F(\widetilde{u}_n))\right]dx +\I\mb R\left[\frac12 g(\widetilde{v}_n)\widetilde{v}_n-F(\widetilde{v}_n))\right]dx\\
	\geq & \I\mb R\left[\frac12 f(\widetilde{u})\widetilde{u}-F(\widetilde{u}))\right]dx +\I\mb R\left[\frac12 g(\widetilde{v})\widetilde{v}-F(\widetilde{v}))\right]dx\\
	=&J(\widetilde{w})-\frac12 \ld \mc J'(\widetilde{w}),\widetilde{w}\rd +o_n(1)\\
	=&J(\widetilde{w}) +o_n(1).
	\end{align*}
On the other hand obviously $\mc J(\widetilde{w})\geq m^\ast=\displaystyle\inf _{\ell\in\mc N} \mc J(\ell)$. Thus $\widetilde{w}$ is a ground state solution of \eqref{s1.4}. This completes the proof. \QED
\section{A priori estimates}\label{s4}
In this section, we obtain a priori estimates for any $(u,v) \in \mc S$ and complete the proof of Theorem \ref{t1.3}. To achieve this, we  prove a series of preliminary results.
 \begin{Proposition}\label{p4.1}
 	Let $(u,v)\in \mc S$. Then $(u,v) \in L^\infty(\mb R)\cap C^{0,\gamma}_{\text{loc}}(\mb R)$ for every $\gamma \in (0,1)$.
 \end{Proposition}
\begin{proof}
For any $r>0$, note that $v$ is a weak solution of the following problem
\begin{equation*}
	(-\De)^\frac12 V +V_0V =f(u),~V-v \in H^\frac12(B_{2r}(0)).
\end{equation*}
By the Trudinger-Moser inequality \eqref{eTME} , we see that $f(u) \in L^p(B_{2r}(0))$, for any  $p \in [2,\infty)$. Now by \cite[Corollary 3.1]{MSY} we have 
\begin{equation*}
	\|(-\De)^\frac14 v\|_{L^p(B_r(0))}\leq C(\|f(u)\|_{L^p(B_{2r}(0))}+\|v\|_\frac12),
\end{equation*}
and so 
\begin{equation}\label{4.7}
	\|v\|_{W^{\frac12,p}(B_r(0))} \leq C(\|f(u)\|_{L^p(B_r(0))}+\|v\|_\frac12),
\end{equation}
where $C>0$ is a constant independent of $v$. Since $p \in [2,\infty)$, we have from Sobolev embeddings that $v \in C^{\gamma}(B_r(0))$ for every $\gamma \in (0,1)$ and
\begin{equation}\label{4.8}
	\|v\|_{C^\gamma(\overline{B_r(0)})} \leq C'\|v\|_{W^{\frac12,p}(B_r(0))},
\end{equation}
where again $C'>0$ is independent of $v$ and obviously $\gamma$ is dependent on $p$. From now, we will work with fixed $p$ and $\gamma$. Combining \eqref{4.7} and \eqref{4.8}, we have
\begin{equation}\label{4.9}
		\|v\|_{C^\gamma(\overline{B_r(0)})}\leq C(\|f(u)\|_{L^p(B_r(0))}+\|v\|_\frac12),
\end{equation}
for some positive constant $C$. Next we prove that $\ds \lim_{|x| \ra \infty} v(x)=0$. Indeed, otherwise there exists $\{x_n\}\subset \mb R$ with $|x_n| \ra \infty$, as $n \ra \infty$ and $\liminf_{n \ra \infty} |v(x_n)|>0$. Let $u_n(x)=u(x+x_n)$ and $v_n(x)=v(x+x_n)$. Then $\|v_n\|_\frac12 =\|v\|_\frac12$ and 
\begin{equation}\label{4.10}
	(-\De)^\frac12v_n +V_0v_n =f(u_n),~ u_n, ~v_n \in H^\frac12(\mb R).
\end{equation}
Assume that $v_n \rightharpoonup \hat{v} $ weakly in $H^\frac12(\mb R)$. We claim that $\hat{v} \not \equiv 0$. Indeed, since $v_n$ is a weak solution of \eqref{4.10} replacing $f(u)$ by $f(u_n)$, it follows from \eqref{4.9} that, up to a subsequence, $v_n \ra \hat{v}$ uniformly in $\overline{B_r(0)}$. Hence
\begin{align*}
	\hat{v}(0)=\liminf\limits_{n \ra \infty }v_n(0) = \liminf\limits_{n \ra \infty }v(x_n) \ne 0.
	\end{align*}
On the other hand, for any fixed $R>0$ and $n$ large enough, we have
\begin{align*}
	\I\mb R v^2 dx \geq \I{B_R(0)}v^2dx +\I{B_R(x_n)}v^2dx
	{=}& \I{B_R(0)}v^2dx +\I{B_R(0)}v_n^2dx\\
	=&\I{B_R(0)}v^2dx +\I{B_R(0)}\hat{v}^2dx +{o_n(1)},
\end{align*}
where $o_n(1)\ra 0$, as $n \ra \infty$. Since $R$ is arbitrary, we get $\hat{v} =0$, which is a contradiction. Thus $v(x) \ra 0$ as $|x| \ra \infty$. Moreover, since $v \in C(B_r(0))$, for any $r>0$, we have $v \in L^\infty(\mb R)$. Similarly, we get $u \in L^\infty (\mb R)$.\QED 
\end{proof}
\begin{Proposition}\label{pC}
	Let $x_w \in \mb R$ be the maximum point of $|u(x)|+|v(x)|$, then the set 
	\begin{align*}
		\{(u(\cdot + x_w),~v(\cdot+x_w))| ~(u,v)\in \mc S\}
	\end{align*}
	is compact in $H^\frac12(\mb R)\times H^\frac12(\mb R)$.
\end{Proposition}
\begin{proof}
	Let $\{w_n\}\subset \mc S$. Then we have the following
	\begin{equation*}
		\mc J(w_n)=m^\ast~\text{and}~\mc J'(w_n)=0,~\text{for all}~n \in \mb N.
	\end{equation*} 
	Thus {as in the proof of} Theorem \ref{t1.1}, we see that there exists $\{y_n\}\subset \mb R$ and $w_0 \not \equiv 0$ such that 
	\begin{equation*}
		w_n(\cdot+y_n) \rightharpoonup w_0~\text{weakly in}~W~\text{and}~w_n(\cdot+y_n)\ra w_0~\text{a.e in}~\mb R,~\text{as}~n \ra \infty.
	\end{equation*}
	Also it is easy to show that $w_0$ is a critical  point of $\mc J$. Now to complete the proof, we show that $w_0 \in \mc S$ and $w_n(\cdot+y_n) \ra w_0$ strongly in $W$, as $n \ra \infty$, and hence $\mc S$ is a compact set. To this end, first note that $\mc J$ is invariant under translation. Thus for sake of simplicity, let us write $w_n$ in place of $w_n(\cdot+y_n)$ and let $w_n=(u_n,v_n)$ and $w_0=(u_0,v_0)$. Now we prove that $w_n \ra w_0$ in $W$. By $(H_3)$ and Fatou's lemma, we have
	\begin{align}\nonumber\label{4.5}
		m^\ast=& \lim\limits_{n \ra \infty} \left( \mc  J(w_n)-\frac12 \ld \mc J'(w_n),w_n \rd\right)\\ \nonumber
		\geq & \lim\limits_{n \ra \infty}\left( \I\mb R\left(\frac12 f({u}_n){u}_n-F({u}_n)\right)dx +\I\mb R\left(\frac12 g({v}_n){v}_n-G({v}_n)\right)dx\right)\\ \nonumber
		\geq & \I\mb R\left(\frac12 f({u}_0){u}_0-F({u}_0)\right)dx+ \I\mb R\left(\frac12 g({v}_0){v}_0-G({v}_0)\right)dx\\
		=&\mc  J(w_0)-\frac12 \ld \mc J'(w_0),w_0 \rd= \mc J(w_0).
	\end{align}
On the other hand, since $w_0 \not \equiv 0$ and $\mc J^\prime(w_0)=0$, one has $\mc  J(w_0) \geq m^\ast$. Thus $w_0 \in \mc S$. Next we prove that $w_n \ra w_0$ in $W$. By \eqref{4.5} and $\mc J(w_0)= m^\ast$, we have
	\begin{equation*}
		\lim\limits_{n \ra \infty}\I\mb R\left(\frac12 f({u}_n){u}_n-F({u}_n)\right)dx= \I\mb R\left(\frac12 f({u}_0){u}_0-F({u}_0)\right)dx
	\end{equation*}
	and
	\begin{equation*}
		\lim\limits_{n \ra \infty}\I\mb R\left(\frac12 g({v}_n){v}_n-G({v}_n)\right)dx= \I\mb R\left(\frac12 g({v}_0){v}_0-G({v}_0)\right)dx.
	\end{equation*}
	Now by $(H_3)$, we get 
	\begin{align*}
		0\leq \frac{\mu-2}{2}F(u_n)\leq \frac12f(u_n)u_n-F(u_n),~0\leq \frac{\mu-2}{2}G(v_n)\leq \frac12g(v_n)v_n-G(v_n).
	\end{align*}
	Thus by the dominated convergence theorem, we have
	\begin{align*}
		\lim\limits_{n \ra \infty} \I\mb R F(u_n)dx= \I\mb R F(u_0)dx,~\lim\limits_{n \ra \infty} \I\mb R G(v_n)dx= \I\mb R G(v_0)dx.
	\end{align*}
	and so 
	\begin{align}\label{4.2}
		\lim\limits_{n \ra \infty}\I\mb R\frac12 f({u}_n){u}_ndx=\I\mb R\frac12 f({u}_0){u}_0dx,~\lim\limits_{n \ra \infty}\I\mb R\frac12 g({v}_n){v}_ndx=\I\mb R\frac12 g({v}_0){v}_0dx. 
	\end{align}
	Now from $\ds \ld J'(u_n,v_n),(0,u_n)\rd =0$ and $\ds \ld J'(u_0,v_0),(0,u_n)\rd =0$, we obtain
	\begin{equation}\label{e4.3}
		\|u_n\|_\frac12^2 =\I\mb Rg(v_n)u_ndx ~\text{and}~	\ld u_n,u_0\rd_\frac12 =\I\mb Rg(v_0)u_ndx.
	\end{equation}
	By using the fact that $u_n \rightharpoonup u_0$ weakly in $H^\frac12(\mb R)$, we conclude that
	\begin{equation}\label{e4.9}
		\lim\limits_{n \ra \infty}\I\mb R g(v_0)u_ndx =\|u_0\|_\frac12^2.
	\end{equation}
	Combining \eqref{e4.3} and \eqref{e4.9} to prove the required convergence, it is sufficient to show that 
	\begin{equation}\label{e4.5}
		\lim\limits_{n\ra \infty}\I\mb R (g(v_n)-g(v_0))u_ndx =0.
	\end{equation}
	We divide the proof of \eqref{e4.5} into two steps:\\
	\textbf{Step 1.} In this step we prove
	\begin{equation}\label{e4.6}
		\I\mb R |g(v_n)-g(v_0)|dx \ra 0,~\text{as}~n \ra \infty.
	\end{equation}
	Using Proposition \ref{p3.1} and Vitali's convergence theorem, we can see that $\{g(v_n)\}$ is relatively compact in $L^1_{\text{loc}}(\mb R)$. Thus to prove \eqref{e4.6} it is sufficient to show that there exists $R'>0$ large enough such that
	\begin{equation*}
		\I{\mb R \setminus B_{R'}(0)}|g(v_n)| dx\leq \e.
	\end{equation*}
	Indeed using $(H_2       )$ there exists $\de >0$ such that
	\begin{equation}\label{eq4.11}
		|g(t)| \leq  \e t^2,~\text{for}~|t| \leq \de.
	\end{equation}
	{Now for any $R>0$ and $K>0$, using Proposition \ref{p3.1} and \eqref{eq4.11} we have
	\begin{align} \nonumber\label{eq4.12}
		\I{\mb R \setminus B_{R}(0)}|g(v_n)|dx =&\I{\{x \in \mb R\setminus B_{R}(0):
			~|v_n(x)|\leq \de\}}|g(v_n)| dx+\I{\{x \in \mb R\setminus B_{R}(0):~\de\leq|v_n(x)|\leq K\}}|g(v_n)|dx\\ \nonumber
		&+\I{\{x \in \mb R\setminus B_{R}(0):~|v_n(x)|\geq K\}}|g(v_n)|dx\\
		&\leq \e \|v_n\|_{L^2(\mb R)}^2 + \I{\{x \in \mb R\setminus B_{R}(0):~\de\leq|v_n(x)|\leq K\}}|g(v_n)|dx +\frac{C}{K}.
	\end{align}
Again using \eqref{4.2} for any $\e^\prime >0$, we have for $n$ large
\begin{align}\label{eq4.13}
	 \I{\{x \in \mb R\setminus B_{R}(0):~\de\leq|v_n(x)|\leq K\}}|g(v_n)|dx \leq \frac{1}{\de} \I{\mb R \setminus B_R(0)}g(v_n)v_ndx \leq  \frac{1}{\de} \left( \I{\mb R \setminus B_R(0)}g(v_0)v_0dx + \e^\prime\right).
 \end{align}
From \eqref{eq4.12}, \eqref{eq4.13} and again employing Proposition \ref{p3.1}, we conclude for $R^\prime>0$,  $K^\prime>0$ and $n$ large enough that
\begin{equation*}
		\I{\mb R \setminus B_{R^\prime}(0)}|g(v_n)|dx \leq  \e.
\end{equation*}
	This completes Step 1.}\\
	\textbf{Step 2.} In this step we complete the proof of \eqref{e4.5}. Now for any $K>0$, we have
	\begin{align*}
		\left| \I\mb R (g(v_n)-g(v_0))u_ndx\right|\leq K \I{|u_n| \leq K}|g(v_n)-g(v_0)|dx +\I{|u_n| \geq K}|g(v_n)-g(v_0)||u_n|dx
		\leq K I_1 +I_2,
	\end{align*}
	where
	\begin{equation*}
		I_1 =\I\mb R|g(v_n)-g(v_0)|dx~\text{and}~I_2=\I{\{|u_n| \geq K\}}|g(v_n)-g(v_0)||u_n|dx.
	\end{equation*}
	Using Step 1, we see that $I_1 \ra 0$ as $n \ra \infty$. Now we estimate $I_2$. We have
	\begin{align*}
		I_2 \leq &\I{\{ |u_n|\geq K, |v_n|\leq \de \}}|g(v_n)u_n|dx +\I{\{ |u_n|\geq K, \de\leq |v_n|\leq K' \}}|g(v_n)u_n|dx +\I{\{ |u_n|\geq K, |v_n|\geq K' \}}|g(v_n)u_n|dx\\
		&+\frac{C}{K}\|u_n\|_{L^2(\mb R)}^2,
	\end{align*}
	where $K' >0$ and $\de >0$ is chosen so that by $(H_1)$, we have
	\begin{equation*}
		g(t)\leq \e t,~\text{for}~|t|\leq \de.
	\end{equation*}
Thus by $(H_1)$, H\"{o}lder's inequality and Proposition \ref{p3.1}, we have
\begin{equation*}
	I_2 \leq C\e +\frac{2C}{K}+ I_{2,1},
\end{equation*}	
where
 \begin{equation*} I_{2,1}=\I{\{ |u_n|\geq K, |v_n|\geq K' \}}|g(v_n)u_n|dx.
	\end{equation*}
 Lastly, we will estimate $I_{2,1}.$
  We need the following inequality (see \cite[Lemma 4.1]{SoO})
 \begin{equation}\label{e3.26}
 	st \leq t^2(e^{t^2}-1) +|s|(\log|s|)^\frac12,~\text{for all}~t \in \mb R,~|s| \geq e^\frac{1}{\sqrt[3]{4}}.
 \end{equation}
 Since from $(H_6)$, $f$ and $g$ have $\ba_0$-critical growth of Trudinger-Moser type and hence for any fixed $\e >0$, there exists $C_\e>0$ such that
 \begin{equation}\label{e3.29}
 	|f(t)|,~|g(t)| \leq C_\e e^{(\ba_0+\e)t^2},~\text{for}~t \in \mb R.
 \end{equation}
By Proposition \ref {pp3.4}, choosing $\de >0$ such that $m^\ast \in \left(0,\frac{\pi}{\ba_0}-\de\right)$ and taking 
 \begin{equation*}
 	\hat{u}_n =\left(\frac{\pi}{\ba_0}-\de\right)^\frac12 \frac{u_n}{\|u_n\|_\frac12},
 \end{equation*}
 we can write
 \begin{align}\label{gep}
 	\I{\{ |u_n|\geq K, |v_n|\geq K' \}}|g(v_n)u_n|dx = \frac{\|u_n\|_\frac12}{\left(\frac{\pi}{\ba_0}-\de\right)^\frac12}\I{\{ |u_n|\geq K, |v_n|\geq K' \}} |g(v_n)||\hat{u}_n|dx=\frac{\|u_n\|_\frac12}{\left(\frac{\pi}{\ba_0}-\de\right)^\frac12} I_{2,1}^\prime,
\end{align}
where 
\begin{equation*}
	I_{2,1}^\prime =\frac{C_\e}{\sqrt{\ba_0}}\I{\{ |u_n|\geq K, |v_n|\geq K' \}} \frac{g(v_n)}{C_\e}\sqrt{\ba_0}|\hat{u}_n|dx.
\end{equation*}
Setting 
\begin{equation*}
	A_n:= \{x \in \mb R:|u_n(x)| \geq K,~|v_n(x)| \geq K',~ |g(v_n(x))|/C_\e \geq e^{\frac{1}{\sqrt[3]{4}}}\}
\end{equation*}
and
\begin{equation*}
	B_n=\{x \in \mb R:|u_n(x)| \geq K,~|v_n(x)| \geq K',~ |g(v_n(x))|/C_\e \leq e^{\frac{1}{\sqrt[3]{4}}}\},
\end{equation*}	
and using inequality \eqref{e3.26} with $s=g(v_n)/C_\e$ and $t =\sqrt{\ba_0} |\hat{u}_n|$, we can estimate
\begin{align*}
	I_{2,1}^\prime \leq \frac{C_\e}{\sqrt{\ba_0}}\I {A_n} \frac{|g(v_n)|}{C_\e}\left[\log\left(\frac{|g(v_n)|}{C_\e}\right)\right]^\frac12dx+\I{B_n}|g(v_n)\hat{u}_n|dx+C_\e\sqrt{\ba_0} \I{A_n} \hat{u}_n^2 (e^{\ba_0\hat{u}_n^2}-1)dx.
\end{align*}
Using \eqref{e3.29}, we get
\begin{align}\label{e4.13}
I_{2,1}^\prime &\leq \sqrt{\frac{\ba_0+\e}{\ba_0}}\I {A_n} g(v_n)v_ndx+\I{B_n}|g(v_n)\hat{u}_n|dx
+C_\e\sqrt{\ba_0} \I {A_n} \hat{u}_n^2 (e^{\ba_0\hat{u}_n^2}-1)dx.
\end{align}	
Using \eqref{4.2}, we see that the first integral is uniformly small for $K'$ large enough. Now
\begin{align*}
\frac{\|u_n\|_\frac12}{\left(\frac{\pi}{\ba_0}-\de\right)^\frac12}	\I{B_n}|g(v_n)||\hat{u}_n|dx \leq C_\e e^{\frac{1}{\sqrt[3]{4}}}\I{B_n}|u_n|dx \leq \frac{C}{K}\I{B_n}|u_n|^2 dx\leq \frac{C^{\prime\prime}}{K}
\end{align*}
which can be made sufficiently small for $K$ large enough. Lastly, we have
\begin{align*}
	\I{A_n}\hat{u}_n^2 (e^{\ba_0\hat{u}_n^2}-1)dx\leq\frac{1}{K} \I{A_n}\hat{u}_n^3 (e^{\ba_0\hat{u}_n^2}-1)dx.
\end{align*}
Since $ \ds \|\hat{u}_n\|_\frac12^2 =\frac{\pi}{\ba_0}-\de$, there exists $p>1$ such that $\ds p\ba_0 \left(\frac{\pi}{\ba_0}-\de\right)<\pi$. Thus by Trudinger-Moser inequality, as $n \ra \infty$,
\begin{equation*}
\I{A_n}\hat{u}_n^2 (e^{\ba_0\hat{u}_n^2}-1)dx	\leq \frac{1}{K}\left(\I\mb R |\hat{u}_n|^{3q}dx\right)^\frac1q\left(\I\mb R(e^{p\ba_0\hat{u}_n^2}-1)dx\right)^\frac1p 
\end{equation*}
where $\ds \frac{1}{p}+\frac{1}{q}=1$,	
which can be made arbitrarily small for $K$ sufficiently large. Hence using \eqref{gep}, \eqref{e4.13} and Proposition \ref{p3.1} we see that $I_{2,1}$ can be made arbitrarily small for $K$ and $K'$ large enough. This completes the proof of \eqref{e4.5}. Thus
\begin{equation*}
	\|u_n\|_\frac12 \ra \|u_0\|_\frac12,~\text{as}~n \ra \infty
\end{equation*}	
	and hence $u_n \ra u_0$ strongly in $H^\frac12(\mb R)$. Similarly, $v_n \ra v_0$ strongly in $H^\frac12(\mb R)$. This proves that $\mc S$ is compact up to translation.\QED
\end{proof}
\begin{Proposition}\label{p4.2}
	Let $w_n=(u_n,v_n) \subset \mc S$ such that $\hat{w}_n=w_n(\cdot +y_n) \ra w_0=(u_0,v_0) \in \mc S$ in $W$, then
	\begin{equation*}
		\sup_{n \geq 1} (\|u_n\|_\infty + \|v_n\|_\infty) <\infty.
	\end{equation*}
\end{Proposition}
\begin{proof}
	Let $\hat{u}_n=u(\cdot +y_n)$, $\hat{v}_n =v_n(\cdot+y_n)$. Similarly as above, $\hat{v}_n$ is a weak solution of the following problem
	\begin{equation*}
		(-\De)^\frac12V+V_0 V= f(\hat{u}_n)~\text{in}~B_{2r}(0),~V-\hat{v}_n \in H_0^\frac12(B_{2r}(0)).
	\end{equation*}
Then similar to \eqref{4.9}, we have
\begin{equation*}
	\|\hat{v}_n\|_{C^\gamma(\overline{B_r(0)})}\leq C(\|f(\hat{u}_n)\|_{L^p(B_r(0))}+\|\hat{v}_n\|_\frac12).
\end{equation*}
The rest of the proof follows similarly as the proof of \cite[Propositon 2.6]{CZ} (after equation $(3.34)$ there).
	\end{proof}
\begin{Proposition}\label{p4.3}
	The following a priori estimates hold
	\begin{equation*}
		0<\inf\limits_{w=(u,v) \in  \mc S} \min\{\|u\|_\infty,\|v\|_\infty\}<\sup\limits_{w=(u,v) \in \mc S}(\|u\|_\infty +\|v\|_\infty)<\infty.	\end{equation*}
\end{Proposition}
\begin{proof}
	The proof is similar to the proof of \cite[Proposition 2.7]{CZ} and hence omitted.
\QED
	\end{proof}
Finally following the proof of \cite[Proposition 2.8]{CZ}, we can easily prove the following proposition:
\begin{Proposition}\label{p4.4}
	Let $x_w \in \mb R$ be a maximum point of $|u(x)|+|v(x)|$, $w=(u,v) \in \mc S$. Then $u(x+x_w) \ra 0$ and $v(x+x_w) \ra 0$, as $|x| \ra \infty$, uniformly for any $(u,v) \in\mc S. $
\end{Proposition}
\section{Poho\v{z}aev type identity}\label{s5}
In this subsection, we assume $f,g$ are locally lipschitz and  establish a Poho\v{z}aev type identity \eqref{e5.1} for the system \eqref{s1.4} following the approach in \cite{BMS}.  
 For $\e \in (0,1]$ we construct a family of functions $\{w_\e\}_\e$ which approximate $w =(u,v)\in \mc S$. For $\e \in (0,1]$  $w_\e=(u_\e, v_\e)$ is the solution of the following system
	\begin{align}\label{pe5.2}
	\begin{cases}
		(-\Delta)^\frac12 u_\e+V_0 u_\e=g_\e\;\text{in}\;\mb R,\\
		(-\Delta)^\frac12 v_\e+ V_0 v_\e=f_\e\;\text{in}\;\mb R.\\
	\end{cases}
\end{align}
where the choice of $f_\e$, $g_\e$ is provided by the following lemma (see \cite[Lemma 3.1]{BMS}):
\begin{Lemma}\label{pl5.2}
	Let $ \mu$ be a standard compactly supported positive mollifier and for $\e \in (0,1]$ assume 
	\begin{equation*}
		\mu_\e (x)=\frac{1}{\e}\mu\left(\frac{x}{\e}\right),~x \in \mb R.
	\end{equation*}
Again let $\hslash \in C_0^\infty(B_1)$ be a function such that 
\begin{equation*}
	0\leq \hslash \leq 1,~\hslash\equiv 1~\text{on}~B_\frac{1}{2},~|\hslash'|\leq C.
\end{equation*}
For every $\e \in (0,1]$, we define $\hslash_\e(x)=\hslash(\e x)$, for $x \in \mb R$. Now define the operator $ \mc H_\e: H^\frac{-1}{2}(\mb R) \ra H^\frac{-1}{2}(\mb R)$ as
\begin{equation*}
	\ld \mc H_\e (\Lambda), \psi\rd:= \ld \Lambda, \hslash_\e(\psi \ast \mu_\e) \rd,~\text{for every }~\psi \in H^\frac12(\mb R).\end{equation*}
Then $ \mc H_\e(\Lambda) \in C_0^\infty (\mb R)$ for every $\Lambda \in H^\frac{-1}{2}(\mb R)$ and 
\begin{equation*}
	\lim\limits_{\e \ra 0} \|\mc H_\e(\Lambda)-\Lambda\|_\frac{-1}{2} =0.
\end{equation*}
\end{Lemma}
 Recalling that $f(u),~g(v) \in L_{\text{loc}}^1(\mb R) \cap H^\frac{-1}{2}(\mb R)$ and using the above lemma we define 
 \begin{equation*}
 	f_\e=\mc H_\e(f(u))=\left(f(u)\hslash_\e\right)\ast \mu_\e~\text{and}~g_\e=\mc H_\e(g(v))=\left(g(v)\hslash_\e\right)\ast \mu_\e.
 \end{equation*}
Then we have the following
	\begin{equation}\label{eq5.2}
		f_\e \ra f(u)~\text{and}~g_\e \ra g(v)~\text~\text{strongly in }~ H^\frac{-1}{2}(\mb R),
	\end{equation}
\begin{equation}\label{eq5.3}
	\sup\limits_{\e \in (0,1]}(\|f_\e\|_\infty+ \|g_\e\|_\infty)\leq C<+\infty,~f_\e \ra f(u),~g_\e \ra g(v)~\text{in}~L_{\text{loc}}^p(\mb R),~\text{for every}~p \geq 1, 
\end{equation}
\begin{equation}\label{eq5.4}
\text{and}~	\sup\limits_{\e \in (0,1]}(\|f_\e\|_{\frac12}+ \|g_\e\|_{\frac12})\leq D<+\infty .
\end{equation}
Equation \eqref{eq5.2} follows from Lemma \ref{pl5.2}, while equation \eqref{eq5.3} follows using the fact that $u,v \in L^\infty(\mb R^n)$. Now to prove \eqref{eq5.4} we use the  Fourier transform (denoted by $\mc F $). We first note that for $\xi \in \mb R$:
\begin{align*}
	|\xi|^\frac12|\mc F (f_\e)(\xi)| =|\xi|^\frac12|\mc F ((f(u)\hslash_\e)\ast \mu_\e)(\xi)|= |\xi|^\frac12|\mc F((f(u)\hslash_\e))(\xi)| |\mc F(\mu_\e)(\xi)| \leq C |\xi|^\frac12|\mc F((f(u)\hslash_\e))(\xi)|,
\end{align*}
where $C>0$ is a constant independent of $\e$. Thus to prove \eqref{eq5.4} {it is sufficient to show that} \eqref{eq5.4} holds for $f(u)\hslash_\e$. Now using \eqref{eq5.3}, that $u \in L^\infty(\mb R)$ and since $f$ is locally Liptschitz we get  that
\begin{align*}
	\|f(u)\hslash_\e\|_\frac12^2 =&\I{\mb R} \I{\mb R} \frac{|f(u(x))\hslash_\e(x)-f(u(y))\hslash_\e(y)|^2}{|x-y|^2}dxdy \\
	\leq& C \I{\mb R} \I{\mb R} \frac{|f(u(x))-f(u(y))|^2}{|x-y|^2}dxdy +C\I{\mb R}\I{\mb R}\frac{|\hslash_\e(x)-\hslash_\e(y)|^2}{|x-y|^2}dxdy\\
	\leq& C \I{\mb R} \I{\mb R} \frac{|u(x)-u(y)|^2}{|x-y|^2}dxdy +C\I{\mb R}\I{\mb R}\frac{|\hslash(x)-\hslash(y)|^2}{|x-y|^2}dxdy <+\infty.
\end{align*} 
Similarly \eqref{eq5.4} holds for $g_\e$.\\
The next proposition gives the existence of the family $\{w_\e\}_\e$.
\begin{Proposition}
	For $\e \in (0,1]$, the minimization problem 
	\begin{equation*}
		j := \inf\limits_{z=(\varphi,\psi) \in W} J_\e(z),
	\end{equation*}
has a unique solution $w_\e\in W$ where
\begin{equation*}
	J_\e(z)=\frac12\|z\|_W^2-\I\mb R (f_\e \psi+g_\e \varphi) dx, ~\text{for every}~z=(\varphi,\psi) \in W.
\end{equation*} Moreover we have
\begin{equation}\label{pe5.4}
	\ld w_\e, z\rd_W -\I\mb R (f_\e \psi+g_\e \varphi) dx =0,
\end{equation}
for every $z=(\varphi,\psi) \in W$.
\end{Proposition}
\begin{proof}
	 We first prove that the functional $J_\e$ is coercive. Indeed, using Young's inequality, we have
	 \begin{align}\nonumber\label{pe5.5}
	 	J_\e(z) \geq& \frac12\|z\|_W^2-\|f_\e\|_\frac{-1}{2}\|\psi\|_\frac12-\|g_\e\|_\frac{-1}{2}\|\varphi\|_\frac12\\
	 	\geq& \frac14 \|z\|_W^2 -4\left(\|f_\e\|_\frac{-1}{2}^2+\|g_\e\|_\frac{-1}{2}^2\right).
	 \end{align}
 This implies that $j\in\mb R$ and that any minimizing sequence $\{z_n\} \subset W$ is bounded.  Then,  the existence and uniqueness of $w_\e$ follows from the weak lower semi-continuity and strict convexity of  $J_\e$.
 Finally, \eqref{pe5.4} results from the first order optimality condition. \QED
 	\end{proof}
 
The next proposition shows that the sequence $\{w_\e\}_\e$ actually converges strongly to $w$ in $W$.
\begin{Proposition}\label{pp5.4}
	We have
	\begin{equation*}
		\lim\limits_{\e \ra 0} \|w_\e -w\|_W =0. 
	\end{equation*}
\end{Proposition}
\begin{proof}
	Note that $\{w_\e\}_\e$ is bounded in $W$. Indeed, since $J_\e(w_\e)\leq J_\e(0)=0$ and using \eqref{pe5.5}, we obtain
	\begin{equation*}
		\|w_\e\|_W^2 \leq 16 \left(\|f_\e\|_\frac{-1}{2}^2 +\|g_\e\|_\frac{-1}{2}^2\right).
	\end{equation*}
Since $f_\e,~g_\e$ are uniformly bounded in $H^\frac{-1}{2}(\mb R)$, we get the claimed bound. Up to an extraction of a subsequence let us denote the weak limit of $w_\e$, $\hat{w}=(\hat{u},\hat{v})$. We first prove that $\hat{w}=w$. For this, note that since $w_\e \rightharpoonup \hat{w}$ in $ W$,  $f_\e \ra f(u)$ and $g_\e \ra g(v)$ in $H^\frac{-1}{2}(\mb R)$, we have
\begin{align*}
 \frac12 \|\hat{w}\|^2_W -\I\mb R (f(u)\hat{v} +g(v)\hat{u})dx\leq& \liminf\limits_{\e \ra 0} \left(\frac12 \|w_\e\|_W^2 -\I\mb R (f(u)v_\e +g(v)u_\e) dx\right)\\
  \leq& \liminf\limits_{\e \ra 0}\left( \frac12 \|w_\e\|_W^2 -\I\mb R (f_\e v_\e +g_\e u_\e) dx\right).
\end{align*}
Now let $\bf h$ = $(\varphi,\psi) \in C_0^\infty (\mb R) \times C_0^\infty (\mb R)$. Using that $\mathrm{argmin} J_\e=w_\e$ and the fact that $f_\e \ra f(u)$ and $g_\e \ra g(v)$ in $L_{\text{loc}}^1(\mb R)$ we infer that
\begin{align*}
	\frac12 \|\hat{w}\|^2_W -\I\mb R (f(u)\hat{v} +g(v)\hat{u})dx \leq& \liminf\limits_{\e \ra 0}\frac12 \|\textbf{h}\|_W^2 -\I\mb R (f_\e \psi +g_\e \varphi) dx\\
	=& \frac12 \|\textbf{h}\|_W^2 -\I\mb R (f(u) \psi +g(v) \varphi) dx.
\end{align*}
By density of $C_0^\infty (\mb R) \times C_0^\infty (\mb R)$ in $W$, the previous estimate implies that $\hat{w}$ minimizes the strictly convex functional
\begin{equation*}
	W \ni\textbf{h} \mapsto \frac12 \|\textbf{h}\|_W^2 -\I\mb R (f(u) \psi +g(v) \varphi) dx.
\end{equation*}
Since $w$ is a critical point of the latter, the strict convexity forces $w=\hat{w}$. Recall that 
\begin{equation*}
	\|u\|_\frac12^2=\I \mb R g(v)udx~\text{and}~\|v\|_\frac12^2=\I \mb R f(u)vdx.
\end{equation*}
Similarly from \eqref{pe5.4}, we have
\begin{equation*}
	\|u_\e\|_W^2 =\I\mb R g_\e u_\e dx~\text{and}~\|v_\e\|_W^2 =\I\mb R f_\e v_\e dx.
\end{equation*}
By the strong convergence of $f_\e$ to $f(u)$ and $g_\e$ to $g(u)$ in $H^\frac{-1}{2}(\mb R)$, we conclude that 
\begin{align*}
	\|w\|_W^2 =\I \mb R (f(u)v+g(v)u)dx =\lim\limits_{\e \ra 0} \I \mb R (f_\e v_\e +g_\e u_\e)dx=\lim\limits_{\e \ra 0} \|w_\e\|_W^2.
\end{align*}
Finally the conclusion follows using the uniform convexity of the norm. \QED
\end{proof}
 \subsection{Regularity for the approximating problem}
In this subsection we will prove some regularity estimates for the solutions $w_\e$ of \eqref{pe5.2}. We first recall the following useful inequality:
\begin{Lemma}\cite[Lemma C.1]{BLP}\label{pl5.5}
	Let $1<p<\infty$ and $\ba \geq 1$. For every $a,~b,~m \geq 0$ there holds
	\begin{equation*}
		|a-b|^{p-2}(a-b)(a_m^\ba-b_m^\ba)\geq \frac{\ba p^p}{(\ba+p-1)^p}\left(a_m^\frac{\ba+p-1}{p}-b_m^\frac{\ba+p-1}{p}\right)^p
	\end{equation*}
where we set $a_m =\min\{a,m\}$ and $b_m=\min\{b,m\}$.
\end{Lemma}
\begin{Proposition}\label{pp5.6} 
Let  $\e \in (0,1]$ and $w_\e=(u_\e,v_\e)$ be the solution of \eqref{pe5.2}. Then we have $u_\e,~v_\e \in L^\infty(\mb R)$.	
\end{Proposition}
\begin{proof}
We prove the proposition by using Moser iterations type technique. For every $L>0$, we define
	\begin{equation*}
		h_L(t)=\begin{cases}
			\min\{t,L\},~&\text{if}~t \geq 0,\\
			0,~&\text{if}~t<0.
		\end{cases}
	\end{equation*}
Now define
$	u_{\e,L}=h_L \circ u_\e, ~v_{\e,L}=h_L \circ v_\e \in H^\frac12(\mb R)\cap L^\infty(\mb R).$
Then for every $\ba \geq 1$, taking $(u_{\e,L}^\ba,0)$ as a test function in \eqref{pe5.4} and using Lemma \ref{pl5.5}, we obtain
\begin{align*}
	\frac{4\ba}{(\ba+1)^2}\I\mb R\I \mb R \frac{\left( (u_{\e,L}(x))^\frac{\ba+1}{2}-(u_{\e,L}(y))^\frac{\ba+1}{2}\right)}{|x-y|^2}dxdy \leq \I\mb R |g_\e|u_{\e,L}^\ba dx.
\end{align*}
By the Sobolev inequality, we have for any $q>2$ 
\begin{align*}
	\left(\I\mb R u_{\e,L}^{q\frac{\ba+1}{2}}dx\right)^\frac{2}{q}\leq C\frac{1}{\ba}\left(\frac{\ba+1}{2}\right)^2 \I\mb R |g_\e|u_{\e,L}^\ba dx.
\end{align*}
By using the uniform bound on $|g_\e|$ and $\ba >1$ we obtain
\begin{align*}
	\|u_{\e,L}\|_{L^{q\frac{\ba+1}{2}}}\leq C^\frac{1}{\ba+1}\left(q\frac{\ba+1}{2}\right)^\frac{2}{\ba+1}\|u_{\e,L}\|_{L^\ba}^\frac{\ba}{\ba+1}
\end{align*}
for some constant $C>0$. By setting
	$\ba_{n+1}=q\frac{\ba_n+1}{2},~\ba_0=q,~\zeta_n=\frac{\ba_n}{\ba_n+1}<1,$
we obtain $\ba_n \ra +\infty$ and
\begin{equation*}
	\|u_{\e,L}\|_{\ba_{n+1}}\leq C^\frac{1}{\ba_{n+1}}\ba_{n+1}^\frac{q}{\ba_{n+1}}\|u_{\e,L}\|_{L^{\ba_n}}^{\zeta_n}.
\end{equation*}
On iterating the above inequality, we get for any $n \geq 1$,
\begin{equation*}
	\|u_{\e,L}\|_{L^{\ba_n+1}} \leq C^{\sum\limits_{i=1}^{n+1}\frac{1}{\ba_i}}\left(\prod\limits_{i=1}^{n+1}\ba_i^\frac{1}{\ba_1}\right)^q\|u_{\e,L}\|_{L^q}^{\prod\limits_{i=1}^{n}\zeta_i}.
\end{equation*}
Now $\ba_n$ can be determined explicitly: by setting $\sigma =q/2>1$, that is
	$\ba_n =\sigma^n\ba_0 +\ds\frac{\sigma^{n+1}-\sigma}{\sigma-1}.$
So it holds
\begin{equation*}
	\lim\limits_{n \ra \infty} \frac{\ba_n}{\sigma^n}=\ba_0 +\frac{\sigma}{\sigma-1},
~\text{and}~
	\sum\limits_{i=1}^{+\infty}\frac{1}{\ba_i} < +\infty ~\text{and}~\prod\limits_{i+1}^{+\infty}\ba_i^\frac{1}{\ba_1}<+\infty.
\end{equation*}
Lastly we see
	\begin{align*}\lim\limits_{n \ra \infty} \prod\limits_{i=1}^{n}\sigma_i=	\lim\limits_{n \ra \infty} \prod\limits_{i=1}^{n}\sigma^{n+1}\frac{\ba_0}{\ba_{n+1}}=\frac{q-2}{q-1}.
\end{align*}
Thus we finally have
	$\|u_{\e,L}\|_\infty \leq \ds C\|u_{\e,L}\|_q^\frac{q-2}{q-1}.$
We now let $L \ra \infty$, which gives $u_\e ^+\in L^\infty(\mb R)$. By repeating the argument above for $u_\e^-$, we get $u_\e \in L^\infty (\mb R)$. Similarly we have $v_\e \in L^\infty(\mb R)$. \QED		\end{proof}
\begin{Proposition}\label{p5.6}
	We have $Du_\e,~Dv_\e \in H^\frac12(\mb R)$ for any $\e \in (0,1]$, where $D$ denotes the derivative (with respect to the variable $x \in \mb R$). Moreover
	\begin{equation*}
		\|Du_\e\|_\frac12 \leq \|Dg_\e\|_\frac{-1}{2}~\text{and}~\|Dv_\e\|_\frac12 \leq \|Df_\e\|_\frac{-1}{2}.
	\end{equation*}
\end{Proposition}
\begin{proof}
Let us take $h \in \mb R \setminus \{0\}$. We use the following notations in the course of the proof:\\
If $h \in \mb R\setminus \{0\}$ and $\varphi: \mb R \rightarrow \mb R$ is a measurable function, we set
\begin{equation*}
	\varphi_h(x)=\varphi(x+h)~\text{and} ~ \delta_h \varphi (x)=\varphi(x+h)-\varphi(x).
\end{equation*}	
In case $\varphi: \mb R \times \mb R \rightarrow \mb R$, we use the same notation for 
\begin{equation*}
	\varphi_h(x,y)=\varphi (x+h,y+h)~\text{and}~\de_h \varphi(x,y)=\varphi(x+h,y+h)-\varphi(x,y).
\end{equation*}	
We also set 
\begin{equation*}
\mc G \varphi(x,y) =\frac{\varphi(x)-\varphi(y)}{|x-y|^\frac12}~\text{and}~d\mu=\frac{dxdy}{|x-y|}.	
\end{equation*}
 For the sake of convenience we drop the subscript $\e$.	
Now using the quotient difference method, let $\psi \in H^\frac12(\mb R)$ and take $(\psi_{-h},0)$ as test function in \eqref{pe5.4} and using change of variables, we get 	
\begin{align}\label{pe5.6}
	\I \mb R \I\mb R  \mc G u_h \mc G \psi d\mu + \I\mb R u_h \psi dx -\I\mb R g_h \psi dx=0.
\end{align}	
Also by taking $(\psi,0)$ as test function in \eqref{pe5.4}, we have	
	\begin{align}\label{pe5.7}
		\I \mb R \I\mb R  \mc G u \mc G \psi d\mu + \I\mb R u \psi dx-\I\mb R g \psi dx=0.
	\end{align}
On subtracting \eqref{pe5.7} from \eqref{pe5.6}, we have the following equation:
\begin{equation}\label{pe5.8}
	\I\mb R \I\mb R \de_h \mc G u \mc G\psi d\mu +\I\mb R \de_h u \psi dx-\I\mb R \de_h g \psi dx =0,
\end{equation}	
which holds true for every $\psi \in H^\frac12(\mb R)$. We now test \eqref{pe5.8} with the test function $\psi = \de_h u$, we obtain	
\begin{equation}\label{pe5.9}
	\I\mb R \I\mb R \de_h \mc G u \mc G\de_h u d\mu+\I\mb R |\de_h u|^2dx -\I\mb R \de_h g \de_h udx =0.
\end{equation}	
	Observing that 
	\begin{equation*}
		\mc G(\de_h u) = \frac{u(x+h)-u(y+h)-(u(x)-u(y))}{|x-y|^\frac12}=\de_h \mc G u
	\end{equation*}
and so from \eqref{pe5.9} we obtain
\begin{align*}
	\|\mc G (\de_h u)\|_\frac12^2 =\I\mb R \de_h g \de_h u dx \leq \| \de_h g\|_\frac{-1}{2}\|\de_h u\|_\frac12. 
\end{align*}	
Now divide by $|h|$ and use \cite[Lemma 2.2]{BMS} to get
\begin{equation}\label{pe5.10}
\left\|	\frac{\de_h u}{h}\right\|_\frac12 \leq\|g'\|_\frac{-1}{2}.
\end{equation}	
In particular, we have that
\begin{equation*}
	\sup\limits_{|h|>0} \left\|	\frac{\de_h u}{h}\right\|_{L^2}< +\infty.
\end{equation*}	
Since $\ds\frac{\de_h u}{h} \ra Du$ in the sense of distribution, we get $Du \in L^2(\mb R)$. Moreover, there exists $\{h_n\}_{n \in \mb N} \subset \mb R \setminus \{0\}$ converging to $0$ such that 
\begin{equation*}
	\frac{\de_{h_n}u}{h_n} \ra Du~\text{strongly in}~L^2(\mb R).
\end{equation*}	 
We can thus pass to the limit in \eqref{pe5.10} by using Fatou's Lemma and get the desired result. Similarly the result holds for $v$. \QED	
	\end{proof}
\begin{Remark}\label{r5.7}
	Since by \eqref{eq5.4}, $f_\e, g_\e$ are uniformly bounded in $H^\frac12(\mb R)$, we have $Df_\e,~D g_\e$ are uniformly bounded in $H^\frac{-1}{2}(\mb R)$. This combined with
	Proposition \ref{p5.6} imply that $ Du_\e,~D v_\e$ are uniformly bounded in $H^\frac12(\mb R)$.
\end{Remark}
\begin{Proposition}\label{p5.8}
 We have $Du_\e,~Dv_\e \in L^\infty(\mb R)$, for every $\e \in (0,1]$.
\end{Proposition}
\begin{proof}
   The proof to show $Du_\e \in L^\infty(\mb R)$ follows similarly as the proof of Proposition \ref{pp5.6} by taking the test function $\left(h_L\left(\frac{\de_h u}{h}\right)^\ba,0\right)$ in \eqref{pe5.8}, where $\ba \geq 1$ and $h_L$ are as in Proposition \ref{pp5.6} and using the fact that $Dg_\e \in L^\infty(\mb R)$. Similarly, we get $D v_\e \in L^\infty(\mb R)$. \QED
	\end{proof}
\subsection{End of Theorem \ref{pt5.1}} 
 In this subsection we complete the proof of the Poho\v{z}aev identity \eqref{e5.1}.\\
 \begin{proof}
 Following \cite{BMS}, we will divide the proof into various intermediate steps:\\
 	\textbf{Step 1: Construction of the perturbation}. Let $r>1$ and $\eta \in C_0^\infty(\mb R)$ be a positive cut-off function supported in $B_r$. For $|t| <\de <1$ ($\de$ depending on $\eta$), the map defined by 
 	\begin{equation*}
 		x \mapsto x+  t\eta(x)x :=\bar{x} = H_t(x)
 	\end{equation*} 
 is a smooth diffeomorphism of $\mb R$ which is uniformly bilipschitz for $|t|<\de$, i.e.,
 \begin{equation*}
 	\sup\limits_{|t|<\de} (\| D H_t(x)\|_{L^\infty}+\|  DH_t^{-1}\|_{L^\infty}) <+\infty.
 \end{equation*}	
Since $\partial_t H_t(x)=\eta(x)x$, it holds
\begin{equation*}
	\partial_t (H_t^{-1}(H_t(x))) + DH_t^{-1}(H_t(x))\partial_tH_t(x)=0,
\end{equation*}
so that
\begin{equation}\label{pe5.11}
	\partial_t H_t^{-1}(\bar{x})=-DH_t^{-1}(\bar{x})\eta(H_t^{-1}(\bar{x}))H_t^{-1}(\bar{x}),~~\partial_tH_t^{-1}(\bar{x})_{\vert_{t=0}}=-\eta(\bar{x})\bar{x}.
\end{equation}
Moreover for any fixed $\bar{x}$, we have
\begin{equation}\label{pe5.12}
	\partial_tDH_t^{-1}(\bar{x})=D\partial_tH_t^{-1}(\bar{x})~\text{and}~	\partial_tDH_t^{-1}(\bar{x})_{\vert_{t=0}}=-D(\bar{x}\eta(\bar{x})). 
\end{equation}
Now we define
\begin{equation*}
	u_{\e,t}=u_\e \circ H_t(x)=u_\e(x+t\eta(x)x)~\text{and}~v_{\e,t}=v_\e \circ H_t(x)=v_\e(x+t\eta(x)x).
\end{equation*}
 By using the change of variable $\bar{x}=H_t(x)$, we see that
\begin{align}\nonumber\label{pe5.13}
	\I\mb R \I\mb R \mc Gu_{\e,t} \mc G v_{\e,t}d\mu =& \I\mb R \I\mb R\frac{(u_\e(\bar{x})-u_\e(\bar{y}))(v_\e(\bar{x})-v_\e(\bar{y}))}{|H_t^{-1}(\bar{x})-H_t^{-1}(\bar{y})|^2} DH_t^{-1}(\bar{x})DH_t^{-1}(\bar{y})d\bar{x}d\bar{y}\\
	=& \I\mb R \I\mb R  \mc G u_\e \mc G v_\e\mc S_t \mc R_t^2d\mu,
\end{align}
where we set for $x \ne y$
\begin{equation*}
	\mc R_t(x,y)=\frac{|x-y|}{|H_t^{-1}(x)-H_t^{-1}(y)|},~\mc S_t(x,y)=DH_t^{-1}(x)DH_t^{-1}(y).
\end{equation*}
Observe that, for $|t| <\de$,
\begin{equation}\label{pe5.14}
	\sup\limits_{|t|<\de} (\|\mc S_t\|_{L^\infty} +	 \|\mc R_t\|_{L^\infty}) <+\infty
\end{equation}
and for any $(x,y)\in \mb R^2$, the maps $t \mapsto \mc S_t(x,y)$ and $t \mapsto \mc R_t(x,y)$ are smooth. Also we have
\begin{equation*}
	\partial_t \mc R_t^2=-2\mc R_t^2 \frac{H_t^{-1}(x)-H_t^{-1}(y)}{|H_t^{-1}(x)-H_t^{-1}(y) |} \frac{\partial_t H_t^{-1}(x)-\partial_t H_t^{-1}(y)}{|H_t^{-1}(x)-H_t^{-1}(y) |},
\end{equation*}
\begin{equation*}
	\partial_t \mc S_t= \partial_tDH_t^{-1}(x)DH_t^{-1}(y)+DH_t^{-1}(x)\partial_tD H_t^{-1}(y).
\end{equation*}
According to \eqref{pe5.11} and the bilipschitz character of $H_t$,
\begin{equation*}
\frac{|\partial_t H_t^{-1}(y)-\partial_t H_t^{-1}(y)|}{|H_t^{-1}(x)-H_t^{-1}(y) |} \leq \|DH_t\|_{L^\infty}\text{Lip}(D(H_t^{-1})\eta(H_t^{-1})H_t^{-1})\leq C(\|\eta\|_{C^2})<+\infty.
\end{equation*}
By \eqref{pe5.12} and \eqref{pe5.14}, we infer that $t \mapsto \partial_t \mc S_t$ and $t \mapsto \partial_t \mc R_t$ are continuous for any $(x,y) \in \mb R^2$ and
\begin{equation*}
	\sup\limits_{|t|<\de} \|\partial_t\mc S_t\|_{L^\infty} +	 \|\partial_t\mc R_t\|_{L^\infty} <+\infty. 
\end{equation*}
\textbf{Step 2: Differentiating under the integral sign.} According to \eqref{pe5.13}, our aim is to prove the following chain of equalities for $t =0$
\begin{align}\nonumber\label{pe5.17}
	\I \mb R \I\mb R \frac{d}{dt}\left(\mc G u_{\e,t}\mc G v_{\e,t}\right)d\mu=\frac{d}{dt} 	\I\mb R\I\mb R \mc G u_{\e,t}\mc G v_{\e,t}d\mu=&\frac{d}{dt}\I\mb R \I\mb R \mc G u_\e\mc G v_\e\mc S_t \mc R_t^2d\mu\\
	=&\I\mb R \I\mb R \frac{d}{dt}\left(\mc Gu_\e \mc G v_\e\mc S_t\mc R_t^2\right)d\mu.
\end{align}
First note that since $u_\e$ is Lipschitz, all the integrands above are well defined. We are going to use \cite[Theorem 5.2]{BMS} here to interchange the integral and the differential. In view of this, we claim that the maps
\begin{align*}
	&t \mapsto 	\I\mb R\I\mb R \mc G u_{\e,t}\mc G v_{\e,t}d\mu=\I\mb R \I\mb R  \mc G u_\e \mc G v_{\e} \mc S_t \mc R_t d\mu,\\
	&t \mapsto 	\I \mb R \I\mb R \frac{d}{dt} \mc G u_{\e,t} G v_{\e,t}d\mu,\\
	&t \mapsto \I\mb R \I\mb R \frac{d}{dt}\left( \mc G u_\e \mc G v_\e\mc S_t\mc R_t^2\right)d\mu
\end{align*}
are well defined and continuous in $|t| \leq \de$.
The first map is continuous by the dominated convergence using the smoothness of $\mc S_t$ and $\mc R_t$ and the bound \eqref{pe5.14}. For the second map, using the change of variables $\bar{x}=H_t(x)$, we have
\begin{align}\label{pe5.18}
	\I\mb R \I\mb R \frac{d}{dt}\mc G u_{e,t}\mc G v_{\e,t}d\mu &=\I\mb R\I\mb R (\mc Gu_{\e,t}\mc G (\eta x Dv_\e\circ H_t) +\mc G (\eta x Du_\e\circ H_t)\mc G v_{\e,t} )d\mu\\ \nonumber
	&=\I\mb R \I\mb R  ( \mc G u_\e\mc G(\eta\circ H_t^{-1}H_t^{-1}Dv_\e) +\mc G(\eta\circ H_t^{-1}H_t^{-1}Du_\e)\mc G v_\e)\mc R_t^2 \mc S_t d\mu,
\end{align}
and the integrand is pointwise continuous in $t$, therefore it is sufficient to dominate it uniformly in $|t| \leq \de$. Using \eqref{pe5.14}, we obtain from \eqref{pe5.18} that
\begin{equation*}
	| \mc G u_\e\mc G(\eta\circ H_t^{-1}H_t^{-1}Dv_\e)\mc R_t^2 \mc S_t| \leq C| \mc G u_\e\mc G(\eta\circ H_t^{-1}H_t^{-1}Dv_\e)|.
\end{equation*}
Notice that supp$(\eta\circ H_t^{-1}) \subset B_{R+1}$, hence the last factor is supported in {$\mc A :=(B_{R+1}\times \mb R) \cup (\mb R \times B_{R+1})$}. Note that
\begin{align*}
	| \mc G(\eta\circ H_t^{-1}H_t^{-1}Dv_\e)|\leq&\left| \mc G(\eta \circ H_t^{-1}H_t^{-1})\frac{Dv_\e(x)+Dv_\e(y)}{2}\right|\\
	&+\left|\mc G Dv_\e\frac{\eta(H_t^{-1}(x))H_t^{-1}(x)+\eta(H_t^{-1}(y))H_t^{-1}(y)}{2}\right|\\
	\leq&\|Dv_\e\|_{L^\infty}|\mc G(\eta \circ H_t^{-1}H_t^{-1})|+\|\eta(H_t^{-1})H_t^{-1}\|_{L^\infty}|\mc GD v_\e|.
\end{align*} 
Using the bounds 
\begin{equation*}
	\sup\limits_{|t|<\de}\|\eta(H_t^{-1})H_t^{-1}\|_{L^\infty}=\|\eta x\|_{L^\infty} \leq R
\end{equation*}
and
\begin{equation*}
	\sup\limits_{|t|<\de}|\mc G(\eta\circ H_t^{-1}H_t^{-1})\leq C(\eta,\de)\min\{|x-y|^\frac{-1}{2},|x-y|^\frac12\} \in L^2(\mb R^2,d\mu),
\end{equation*}
we have
\begin{align*}
| \mc G u_\e\mc G(\eta\circ H_t^{-1}H_t^{-1}Dv_\e)\mc R_t^2 \mc S_t| \leq C |\mc G u_\e|_{\chi_{\mc A}}\left(|\mc G v_\e|+\min\{|x-y|^\frac{-1}{2},|x-y|^\frac12\}\right) \in L^1(\mb R^2,d\mu).
\end{align*}
Similarly the second term is dominated by a $L^1$ function.
Thus we get the required continuity due to the dominated convergence theorem.\\
Lastly, we have
\begin{align}\label{pe5.19}
\I\mb R \I\mb R \frac{d}{dt}\left( \mc G u_\e \mc G v_\e\mc S_t\mc R_t^2\right)d\mu 
= \I\mb R \I\mb R  \mc G u_\e \mc G v_\e\left(\partial_t \mc S_t \mc R_t^2+\mc S_t \partial_t \mc R_t^2\right)d\mu.
\end{align}
 Thus we conclude by dominated convergence theorem.\\
Thus \eqref{pe5.17} is proved for $|t|<\de$ and the right-hand sides of \eqref{pe5.18} and \eqref{pe5.19} are equal.\\
\textbf{Step 3: Poho\v{z}aev identity for the approximating problem.} We evaluate right hand sides of \eqref{pe5.18} and \eqref{pe5.19} at $t =0$ to get
\begin{align*}
	\I\mb R \I\mb R (\mc Gu_{\e}\mc G(\eta x Dv_\e)+\mc Gv_{\e}\mc G(\eta x Du_\e))d\mu=2\I\mb R \I\mb R \mc G u_\e \mc Gv_\e \left[\frac{x-y}{|x-y|}.\frac{\eta(x)x-\eta(y)y}{|x-y|}-D(\eta(x)x)\right]d\mu.
\end{align*}
Since $(\eta x Dv_\e,\eta x Du_\e)$ is an admissible test function for \eqref{pe5.4}, we obtain the identity
\begin{align}\label{pe5.20}\nonumber
	\I\mb R (g_\e\eta x Dv_\e+f_\e\eta xDu_\e)dx-&nV_0\I\mb R (u_\e \eta xDv_\e+v_\e \eta xDu_\e)dx\\
	=&2\I\mb R \I\mb R \mc G u_\e \mc Gv_\e \left[\frac{x-y}{|x-y|}.\frac{\eta(x)x-\eta(y)y}{|x-y|}-D(\eta(x)x)\right]d\mu.
\end{align}
\textbf{Step 4: Taking the limit.} We now let $\e \ra 0$ in the previous equality. For the right hand side since $\ds \frac{x-y}{|x-y|}.\frac{\eta(x)x-\eta(y)y}{|x-y|}$,  $D(\eta(x)x)$ are bounded and by Proposition \ref{pp5.4}, $\mc G u_\e \ra \mc G u$ and $\mc G v_\e \ra \mc G v$ strongly in $L^2(\mb R^2,d\mu)$, we have
\begin{align*}
\I\mb R \I\mb R \mc G u_\e \mc Gv_\e \left[\frac{x-y}{|x-y|}.
\frac{\eta(x)x-\eta(y)y}{|x-y|}\right]d\mu-\I\mb R &\mc G u_\e \mc Gv_\e D(\eta(x)x)d\mu\\
 \ra& \I\mb R \I\mb R \mc G u \mc Gv \left[\frac{x-y}{|x-y|}.\frac{\eta(x)x-\eta(y)y}{|x-y|}-D(\eta(x)x)\right]d\mu.
\end{align*}
Now using integration by parts, we have
\begin{align*}
\lim\limits_{\e \ra 0}	\I\mb R (u_\e \eta x Dv_\e+v_\e\eta x Du_\e)dx  = \lim\limits_{\e \ra 0} \I\mb RD(u_\e v_\e)\eta xdx&=-\lim\limits_{\e \ra 0} \I\mb R u_\e v_\e D(\eta(x)x)dx\\
&=-\I\mb R u v D(\eta(x)x)dx.
\end{align*}
Now to compute $\ds \I\mb R g_\e\eta x Dv_\e dx$, we split it as
\begin{align*}
\I\mb R g_\e \eta x Dv_\e dx= \I\mb R (g_\e-g(v_\e))\eta x Dv_\e dx + \I\mb R g(v_\e)\eta x Dv_\e dx.
\end{align*}
Again using integrating by parts, we get
\begin{align}\label{e5.23}
	\lim\limits_{\e \ra 0}\I\mb R g(v_\e)\eta x Dv_\e dx= -\I\mb R G(v)D(\eta(x)x) dx.
\end{align}
For the other term, note that $v_\e$ is uniformly bounded by Proposition \ref{pp5.6} and $g$ is continuous. Since by Proposition \ref{pp5.4} we have $v_\e \ra v $ in $L^p(\mb R)$ for every $ p \in [2,\infty)$, $g_\e - g(v_\e) \ra 0$ in $L^p_{\text{loc}}(\mb R)$. Indeed for every compact set $K \subset \mb R$ we have
\begin{align*}
	\|g_\e-g(v_\e)\|_{L^p(K)} \leq 	\|g_\e-g(v)\|_{L^p(K)}+	\|g(v)-g(v_\e)\|_{L^p(K)}.
	\end{align*}
The first term of the right hand side of the above expression converges to $0$ thanks to Lemma \ref{pl5.2}, while by applying the dominated convergence theorem the second term tends to $0$.\\
Now we claim that
\begin{align}\label{eq5.23}
	\lim\limits_{\e \ra 0} \I\mb R (g_\e-g(v_\e))\eta x Dv_\e dx=0.
\end{align}
Indeed, using \eqref{eq5.2} and Remark \ref{r5.7}, we get
\begin{align*}
	\left|  \I\mb R (g_\e-g(v_\e))\eta x Dv_\e dx\right| \leq C \|g_\e -g(v_\e)\|_\frac{-1}{2}\|\eta xDv_\e\|_\frac12 \ra 0~\text{as}~\e \ra 0.
\end{align*}
Combining \eqref{e5.23} and \eqref{eq5.23}, we obtain
\begin{align*}
	\lim\limits_{\e \ra 0} \I\mb R g_\e \eta x Dv_\e dx= -\I\mb R G(v)D(\eta(x)x)dx.
\end{align*}
Similarly, we have
\begin{align*}
	\lim\limits_{\e \ra 0} \I\mb R f_\e \eta x Du_\e dx= -\I\mb R F(u)D(\eta(x)x) dx.
\end{align*}
Thus by letting $\e\ra 0$ in \eqref{pe5.20}, we obtain
\begin{align}\label{pe5.21}\nonumber
	\I\mb R (F(u)+G(v))D(\eta(x)x)dx-&V_0\I\mb R uvD(\eta(x)x)dx\\
	=&2\I\mb R \I\mb R \mc G u \mc Gv \left[\frac{x-y}{|x-y|}.\frac{\eta(x)x-\eta(y)y}{|x-y|}-D(\eta(x)x)\right]d\mu.
\end{align}
\textbf{Step 5: Conclusion.} Finally, we take $\eta$ of the form $\eta_R(x)=\phi(x/R)$ with $\phi \in C_0^\infty(B_1)$ positive such that $\phi \equiv 1$ in $B_\frac12$. Clearly
\begin{align*}
	\left|\frac{x-y}{|x-y}.\frac{\eta_R(x)x-\eta_R(y)y}{|x-y|}\right| \leq \|\phi\|_{L^\infty} +\|D\phi\|_{L^\infty},
	\end{align*}
and 
\begin{align*}
	|D(\eta_R(x)x)| \leq \left|\frac1R D\phi\left(\frac{x}{R}\right)x\right| + \left|\phi\left(\frac{x}{R}\right)\right|\leq C(\|\phi\|_{L^\infty}+\|D\phi\|_{L^\infty}).
\end{align*}
Moreover, for any $(x,y) \in \mb R^2$
\begin{align*}
	\frac{x-y}{|x-y|}.\frac{\eta_R(x)x-\eta_R(y)y}{|x-y|}\ra 1,~D(\eta_R(x)x) \ra 1
	\end{align*}
as $R \ra \infty$. Hence, using Dominated Convergence theorem (note that $|\mc G(u)\mc G(v)|$ is in $L^1(\mb R^2, d\mu)$), we can let $ R \ra \infty$ in \eqref{pe5.21} to obtain
\begin{align*}
		\I\mb R (F(u)+G(v)- V_0 uv)dx=0. 
	\end{align*}
as desired. This completes the proof. \QED
\end{proof}

\section{Singularly perturbed system}\label{s6}
In this section we will study the singularly perturbed system \eqref{1.1} and completes the proof of Theorem \ref{t6.1}.

\subsection{Functional setting}
First note that by setting $u(x)=\varphi(\e x),~v(x)=\psi(\e x)$, and $V_\e(x)=V(\e x)$, system \eqref{1.1} is equivalent to 
\begin{align}\label{6.2}
	\begin{cases}	
		(-\De)^\frac12 u +V_\e(x) u =g(v), \\
		(-\De)^\frac12v +V_\e(x)v =f(u).
	\end{cases}
\end{align}
We next consider \eqref{6.2}. We define $H_\e$ as the completion of $C_0^\infty(\mb R)$ with respect to the inner product
\begin{equation*}
	\ld u,v\rd_{\frac12,\e}:=  \frac{1}{2 \pi}\I\mb R \I\mb R \frac{(u(x)-u(y))(v(x)-v(y))}{|x-y|^2}dxdy +\I\mb R V_\e(x)uvdx
\end{equation*}
and the norm
\begin{equation*}
	\|u\|_{\frac12,\e}^2:=\ld u,u\rd_{\frac12,\e},~u,~v \in H_\e.
\end{equation*}
Let $W_\e:=H_\e\times H_\e$ with the inner product
\begin{equation*}
	\ld w_1,w_2\rd_\e :=\ld u_1,u_2\rd_{\frac12,\e}+\ld v_1,v_2\rd_{\frac12,\e},~w_i=(u_i,v_i)\in W_\e,~i=1,2,
\end{equation*}
and the norm $\|w\|_\e^2=\|(u,v)\|_{\e}^2=\|u\|_{\frac12,\e}^2+\|v\|_{\frac12,\e}^2$. We have the orthogonal space decomposition $W_\e=W_\e^+ \oplus W_\e^-$, where
\begin{equation*}
	W_\e^+:=\{(u,u)|~u \in H_\e\}~\text{and}~W_\e^-:=\{(u,-u)|~u \in H_
	\e\}.
\end{equation*}
For each $w=(u,v) \in W_\e$,
\begin{equation*}
	w=w^+ +w^- =((u+v)/2, (u+v)/2)+((u-v)/2,(v-u)/2).
\end{equation*}
The associated energy functional corresponding to \eqref{6.2} is given by 
\begin{equation*}
	\mc J_\e(w):= \ld u,v\rd_\e -\Phi(w),~w=(u,v)\in W_\e,
\end{equation*}
where $\Phi(w)=\ds \I\mb R F(u)+G(v)$. Then $\mc J_\e \in C^1(W_\e,\mb R)$ and
\begin{equation*}
\ld	\mc J'_\e(w),z\rd =\ld w_1,z_1\rd_\e +\ld w_2,z_2\rd_\e-\I\mb R(f(w_1)z_1+g(w_2)z_2)dx,
\end{equation*}
for all $w=(w_1,w_2),~z=(z_1,z_2) \in W_\e$. Moreover, $\mc J_\e$ can be written as follows
\begin{equation*}
	\mc J_\e(w):= \frac12 \|w^+\|_\e^2 -\frac12\|w^-\|_\e^2 -\Phi(w).
\end{equation*}
\begin{Definition}
	We define weak solutions of \eqref{6.2} as the critical points of the associated energy functional $\mc J_\e$.
\end{Definition}
We know that if $w \in W_\e$ is a nontrivial critical point of $J_\e$, then $w \in W_\e \setminus W_\e^-$. Again, in the spirit of \cite{P}, we define the generalized Nehari manifold
\begin{equation*}
	\mc N_\e:=\{w \in W_\e\setminus W_\e^-:\ld \mc J_\e'(w),w\rd=0, \ld \mc J'_\e(w),\varphi\rd=0~\text{for all}~\varphi \in W_\e^- \}.
\end{equation*}
Let
\begin{equation*}
	m^\ast_{\e}:= \inf_{w \in \mc N_\e} \mc J_\e(w),
\end{equation*}
then $m^\ast_{\e}$ is the least energy for system \eqref{6.2}, the so-called ground-state level.\\
For $w \in W_\e \setminus W_\e^-$, set 
\begin{equation*}
	\widetilde{W}_\e(w)=W_\e^- \oplus \mb R^+ w=W_\e^- \oplus \mb R^+ w^+,
\end{equation*}
where $\mb R^+ w:=\{tw: t\geq 0\}$. Following Proposition \ref{p2.6}, we have the following properties of $\mc N_\e$:
\begin{Lemma}\label{l6.3}
	Under the assumptions in Theorem \ref{t6.1}, we have:
	\begin{itemize}
		\item [(1)]For any $w \in \mc N_\e$, $\mc J_\e\arrowvert_{\widetilde{W}_\e(w)}$ admits a unique maximum point which occurs precisely at $w$.
		\item[(2)] For any $w \in W_\e \setminus W_\e^-$, the set $\widetilde{W}_\e(w)$ intersects $\mc N_\e$ at exactly one point $\widetilde{m}_\e(w)$, which is the unique global maximum point of $\mc J_\e \arrowvert_{\widetilde{W}_\e(w)}$.
	\end{itemize}
\end{Lemma}
\subsection{ Lower and upper bounds for $m_\e^\ast$}
\begin{Proposition}\label{p6.4}
Assume that $f$, $g$ satisfy $(H_7)$. Then there exists $m_0$ (independent of $\e$) such that for $\e >0$ sufficiently small
	\begin{equation*}
		m^\ast_{\e} =\inf\limits_{w \in W_\e\setminus W_\e^-} \max\limits_{z \in \widetilde{W}_\e(w)}\mc J_\e(z) \in (m_0,\pi/\ba_0).
	\end{equation*}
\end{Proposition}
\begin{proof}
	The min-max characterization follows as in Corollary \ref{c3.3}. Next we estimate from below and above the critical level $m^\ast_{\e}$. \\
	\textit{Lower bound.} The estimate on the lower bound follows on the similar lines of the proof of\\ \cite[Proposition 3.2]{CZ}.
\\
\textit{Upper bound}. Now we give the estimate for the upper bound. By $(H_7)$ and assuming without loss of generality $V(0)=V_0$, we can fix $r_0>0$ and $\e_0>0$  such that
\begin{equation*}
	\al_0 > \frac{2\pi}{\ba_0 r_0} e^{2r_0/\pi \max_{|x|\leq r_0\e}V(x)},~\e \in (0,\e_0).
\end{equation*}
We then consider the Moser type sequence as defined in \eqref{ms} with $r_1$ replaced by $r_0$. In this case we need a more precise estimate of $\|\omega_n\|_{L^2}^2$. We observe that $\ds\|\omega_n\|_{L^2}^2=4r_0/\log n +o(r_0/\log n)$. Define $b_n(r_0)= 4r_0/\pi +o_n(1)$, where $o_n(1) \ra 0$, as $n \ra \infty$ and let $\hat{\omega}_{n,\e}:= \omega_n/\| \omega_n\|_{\frac12,\e}$. Then $\|\hat{\omega}_{n,\e}\|_{\frac12,\e}=1$ and for $n$ large enough,
\begin{equation*}
	\hat{\omega}_{n,\e}^2 \geq \frac{1}{\pi}(\log n-b_{\e,n}(r_0))~\text{for}~|x| \leq \frac{r_0}{n},
\end{equation*} 
where $b_{\e,n}(r_0)=b_n(r_0)\max_{|x| \leq \e r_0} V(x)$.\\
Then, we can argue similarly as in the proof of Proposition \ref{pp3.4} (see also \cite[Proposition 3.2]{CZ}) to get 
\begin{equation*}
\sup\limits_{w \in \widetilde{W}_\e(\hat{\omega}_{n,\e},\hat{\omega}_{n,\e})} J_\e(w) <\pi/\ba_0
\end{equation*}
and the desired upper bound.
\QED
	\end{proof}
\subsection{Existence of solutions to system \eqref{6.2}} 
Let us define 
\begin{equation*}
	\widetilde{m}_\e:w \in W_\e\setminus W_\e^- \mapsto \widetilde{m}_\e(w)\in \widetilde{W}_\e(w)\cap \mc N_\e.
\end{equation*}
We have the following lemma
\begin{Lemma}\label{l6.5}
	There exists $\de>0$ (independent of $\e$) such that $\|w^+\|_\e \geq \de$ for all $w \in \mc N_\e$. In particular,
	\begin{equation*}
		\|\widetilde{m}_\e(w)\|_\e \geq \de~\text{for all}~w \in W_\e\setminus W_\e^-.
	\end{equation*}
Moreover, for each compact subset $\mc K \subset W_\e\setminus W_\e^-$, there exists a constant $C_{\mc K,\e}$ such that
\begin{equation*}
	\|\widetilde{m}_\e(w)\|_\e \leq C_{\mc K,\e}~\text{for all}~w \in \mc K.
\end{equation*}
\end{Lemma}
\begin{proof}
The proof follows similarly as the proof of Proposition \ref{p3.5} and using the uniform lower bound of $m^\ast_{\e}$ obtained in Proposition \ref{p6.4}.	\QED
	\end{proof}
Let 
\begin{equation*}
	S_\e^+ :=\{w \in W_\e^+:\|w\|_\e=1\},
\end{equation*}
then $S_\e^+$ is a $C^1$-submanifold of $W_\e^+$ and the tangent manifold of $S_\e^+$ at $w\in S_\e^+$ is
\begin{equation*}
	T_w(S_\e^+)=\{w \in W_\e^+:(w,z)_\e=0\}.
\end{equation*}
Define
\begin{equation*}
	m_\e:=\widetilde{m}_\e\arrowvert_{S_\e^+}:S_\e^+\ra \mc N_\e,
\end{equation*}
then by Lemma \ref{l6.5}, $\widetilde{m}_\e$ is continuous and $m_\e$ is a homeomorphism between $S_\e^+$ and $\mc N_\e$.\\
 Define
\begin{equation*}
	\mc F_\e:S_\e^+ \ra \mb R, ~\mc F_\e:=\mc J_\e(m_\e(w)),~w\in S_\e^+,
\end{equation*}
then as in Proposition \ref{p2.11}, for any fixed $\e>0$, we have the following
\begin{Proposition}\label{p6.6}
	We have
	\begin{itemize}
		\item [(a)]$\mc F_\e \in C^1(S_\e^+,\mb R)$ and
		\begin{equation*}
			\langle \mc F_\e'(w),z\rangle=\|m_\e(w)^+\|_\e\langle \mc J_\e'(m_\e(w)),z\rangle~\text{for all}~z \in T_w(S_\e^+).
		\end{equation*}
		\item [(b)] If $\{z_n\}\subset S_\e^+$ is a Palais-Smale sequence for $\mc F_\e$, then $\{m_\e(z_n)\}\subset \mc N_\e$ is a Palais-Smale sequence for $\mc J_\e$. Namely, if $\mc F_\e(z_n)\ra d$ for some $d>0$ and $\|\mc F_\e'(z_n)\|_\ast \ra 0$ as $n \ra \infty$, then $\mc J_\e(m_\e(z_n)) \ra d$ and $\| \mc J_\e'(m_\e(z_n))\| \ra 0$ as $n \ra \infty$, where
		\begin{equation*}
			\|\mc F_\e'(z_n)\|_\ast =\sup_{\substack{\phi \in T_{z_n}(S_\e^+)\\ \|\phi\|_\e=1}}\langle \mc F_\e'(z_n),\phi\rangle
			~\text{and}~
			\|\mc J_\e'(m_\e(z_n))\| =\sup_{\substack{\phi \in W_\e\\ \|\phi\|_\e=1}}\langle {\mc J'(m(z_n))},\phi\rangle.
		\end{equation*}
		\item [(c)] $z \in S_\e^+$ is a critical point of $\mc F_\e$ if and only if $m_\e(z) \in \mc N_\e$ is a critical point of $\mc J_\e$.
		\item [(d)] $\inf_{S_\e^+}\mc F_\e =\inf_{\mc N_\e} \mc J_\e$.
	\end{itemize}
\end{Proposition}
Since $S_\e^+$ is a regular $C^1$-submanifold of $W_\e^+$, by {Propositions \ref{p6.4} and \ref{p6.6}}, it follows from the Ekeland variational principle (see \cite[Theorem 3.1]{E}) that there exists $\{z_n\}\subset S_\e^+$ such that
\begin{equation*}
	\mc F_\e(z_n)\ra m^\ast_\e>0~\text{and}~\|\mc F_\e'(z_n)\|_\ast \ra 0,~\text{as}~n \ra \infty.
\end{equation*} 
Let $w_n=m_\e(z_n)\in \mc N_\e$, then
\begin{equation*}
	\mc F_\e(w_n)\ra m^\ast_\e>0~\text{and}~\|\mc F_\e'(w_n)\|_\ast \ra 0,~\text{as}~n \ra \infty.
\end{equation*} 
Similarly as in Section \ref{s3}, we have the following two propositions:
\begin{Proposition}\label{p6.7}
	There exists $C$ (independent of $\e$) such that for all $\e>0$ and $n \in \mb N$:
	\begin{itemize}
		\item [(1)] $\|w_n\|_\e=\|(u_n,v_n)\|_\e\leq C(1+m^\ast_\e)$;
		\item [(2)] $\I \mb R f(u_n)u_ndx \leq C(1+m^\ast_\e)$ and $\I \mb R g(v_n)v_ndx \leq C(1+m^\ast_\e)$;
		\item [(3)] $\I \mb R F(u_n)dx \leq C(1+m^\ast_\e)$ and $\I \mb R G(v_n)dx \leq C(1+m^\ast_\e)$.
	\end{itemize}
	\end{Proposition}
Up to a subsequence, there exists $w_\e=(u_\e,v_\e)\in W_\e$ such that $w_n \rightharpoonup w_\e$ in $W_\e$ and $w_n \ra w_\e$ a.e. in $\mb R$, as $n \ra \infty$, which is actually a weak solution to \eqref{6.2}. Precisely we have
\begin{Proposition}\label{p6.8}   	The weak limit $w_\e$ is a critical point of $\mc J_\e$.
\end{Proposition}
Now the only thing that remains to show is $w_\e \not \equiv 0$.
For this, we study the asymptotic behaviour of $m^\ast_\e$. More precisely, we investigate the relation between $m^\ast$ and $m^\ast_\e$, where $m^\ast$, $m^\ast_\e$ are the corresponding least energies to systems \eqref{1.1} and \eqref{6.2}, respectively. We have the following result: 
\begin{Lemma}\label{l6.9}
	With the assumptions of Theorem \ref{t6.1}, we have
	\begin{equation*}
		\limsup\limits_{\e \ra 0} m^\ast_\e \leq m^\ast.
	\end{equation*}
\end{Lemma}
\begin{proof}
	The proof follows on the similar lines of the proof of \cite[Lemma 3.7]{CZ}. \QED
\end{proof}
\subsection{Existence of ground state solutions for \eqref{6.2}}
For any $\theta >0$, let us consider the following problem in $\mb R$
\begin{equation}\label{6.6}
	\begin{cases}
		(-\De)^\frac12u +\theta u=g(v),\\
		(-\De)^\frac12 v +\theta v=f(u).
	\end{cases}
\end{equation}
The associated energy functional to \eqref{6.6} is given by 
\begin{equation*}
	\mc J_\theta(w):= \ld u,v \rd +\I\mb R\theta uvdx-\Phi(w),~w=(u,v) \in W.
\end{equation*}
As above, one can define the generalized Nehari manifold $\mc N_\theta$ and the least energy
\begin{equation*}
	m^\ast_\theta:= \inf\limits_{w \in \mc N_\theta} \mc J_\theta(w).
\end{equation*}
Moreover, under the assumptions of Theorem \ref{t1.1}, if $m^\ast_\theta \in (0,\pi/\ba_0)$ for some $\theta>0$, then there exists $w_\theta =(u_\theta,v_\theta) \in \mc N_\theta$ such that $\mc J_\theta(w_\theta)=m^\ast_\theta$. 
\begin{Lemma}\label{l6.10}
 Under the assumptions of Theorem \ref{t1.1}, for any $\theta >0$ the map $\theta \mapsto m^\ast_\theta \in (0,\pi/\ba_0)$ is strictly increasing.
\end{Lemma}
\begin{proof}
	For a proof see proof of  \cite[Lemma 3.8]{CZ} with straightforward modifications.
\QED
	\end{proof}
Now, we show that the weak limit obtained in Proposition \ref{p6.8} is non trivial, precisely
\begin{Lemma}\label{l6.11}
	$w_\e \not \equiv 0$ provided $\e>0$ is sufficiently small.
\end{Lemma}
\begin{proof}
{We proceed by contradiction and suppose $w_\e =0$ for $\e >0$ small, then the associated Palais Smale sequence (see Proposition \ref{p6.6}) $w_n=(u_n,v_n)\rightharpoonup 0$ in $W_\e$ and $w_n \ra 0$ a.e. in $\mb R$, as $n \ra \infty$. Obviously $\{w_n\}$ satisfies one of the following alternatives:\\
	\textit{Vanishing}:
	\begin{equation*}
		\lim\limits_{n \ra \infty} \sup_{y \in \mb R} \I{B_R(y)}(u_n^2 +v_n^2)dx=0~\text{for all}~R>0.
	\end{equation*}
\textit{	Nonvanishing}: there exists $\nu >0$, $R'>0$ and $\{y_n\}\subset \mb R$ such that
	\begin{equation*}
		\lim\limits_{n \ra \infty} \I{B_{R'}(y_n)}(u_n^2+v_n^2)dx \geq \nu.
	\end{equation*}
Due to $m^\ast_\e \in (m_0,\pi/\ba_0$) we can rule out \textit{Vanishing} as in the proof of Lemma \ref{l3.6}. So the \textit{Nonvanishing} occurs. Let $\hat{u}_n(\cdot):=u_n(\cdot +y_n)$ and $\hat{v}_n(\cdot):=v_n(\cdot +y_n)$, then $|y_n| \ra \infty$, as $n \ra \infty$ and
\begin{equation}\label{6.8}
	\lim\limits_{n \ra \infty} \I{B_{R'}(0)}(\hat{u}_n^2+\hat{v}_n^2)dx \geq \nu.
\end{equation}
Let $\hat{w}_n=(\hat{u}_n,\hat{v}_n)$. Then $\{\hat{w}_n\}$ is bounded in $W$. Up to a subsequence, by \eqref{6.8}, we assume that $\hat{w}_n \ra \hat{w} \ne 0$ weakly in $W$ for some $\hat{w}=(\hat{u},\hat{v}) \in W$ and $\mc J'_{V_\infty}(\hat{w})=0$, where
\begin{equation*}
	\mc J_{V_\infty}(w)=\I\mb R\I\mb R \frac{(u(x)-u(y))(v(x)-v(y))}{|x-y|^2}dxdy + V_\infty \I\mb R uv-\Phi(w),~w=(u,v)\in W.
\end{equation*}
By $(H_2)$ and Fatou's lemma, for fixed $\e>0$,
\begin{align*}
	m^\ast_\e+o_n(1)&=\mc J_\e(\hat{w}_n)-\frac12 \ld \mc J'_\e(\hat{w}_n),\hat{w}_n\rd\\
	&=\I\mb R \frac12 f(\hat{u}_n)\hat{u}_n-F(\hat{u}_n) +\I\mb R\frac12 g(\hat{v}_n)\hat{v}_n-G(\hat{v}_n)\\
	&\geq \I\mb R \frac12 f(\hat{u})\hat{u}-F(\hat{u}) +\I\mb R\frac12 g(\hat{v})\hat{v}-G(\hat{v})\\
	&= \mc J_{V_\infty}(\hat{w})-\frac12 \ld \mc J'_{V_\infty}(\hat{w}),\hat{w}\rd+o_n(1)\geq m^\ast_{V_\infty}+o_n(1).
\end{align*}
It follows that $m^\ast_\e \geq m^\ast_{V_\infty}$ for $\e >0$ small enough. By Lemma \ref{l6.10}, we get $m^\ast_{V_\infty}> m^\ast$ for $\e >0$ small enough and by Lemma \ref{l6.9}, we finally get a contradiction.}
 \QED
\end{proof}
By virtue of Lemma \ref{l6.11}, the following is straightforward.
\begin{Corollary}
	For $\e >0$ small enough, $\mc J_\e(w_\e)=m^\ast_\e$, namely, $w_\e$ is ground-state solution of \eqref{6.2}.
\end{Corollary}
\subsection{Concentration}
Reasoning as in Proposition \ref{p4.1}, we have
\begin{Proposition}\label{6.13}
	Let $\e>0$ and $w_\e=(u_\e,v_\e)$ be a ground-state solution to \eqref{6.2}. Then, $u_\e,~v_\e \in L^\infty(\mb R)\cap C_{loc}^\gamma(\mb R)$ for every $\gamma \in (0,1)$. Moreover, $u_\e(x),~v_\e(x) \ra 0$, as $|x| \ra \infty$.
\end{Proposition}
By Proposition \ref{6.13}, there exists $y_\e \in \mb R$ such that
\begin{equation*}
	|u_\e(y_\e)|+|v_\e(y_\e)|=\max_{x\in \mb R} (|u_\e(x)|+|v_\e(x)|).
\end{equation*}
Moreover, $x_\e:=\e y_\e$ is a maximum point of $|\varphi_\e(x)| +|\psi_\e(x)|$, where $(\varphi_\e(\cdot),\psi_\e(\cdot))=(u_\e(\cdot/\e),v_\e(\cdot/\e))$
is a ground state solution of the original problem \eqref{1.1}. We complete the proof of Theorem \ref{t6.1} by proving the two next propositions. The first one follows using similar arguments as in the proof of \cite[Proposition 3.12]{CZ} and hence omitted.
\begin{Proposition} We have
\begin{enumerate}
	\item $\lim_{\e \ra 0}$ dist$(x_\e,\mc L)$ =0;
	\item $(u_\e(\cdot+x_\e/\e),v_\e(\cdot+x_\e/\e))$ converges (up to a subsequence) to a ground-state solution of 
	\begin{equation*}
	\begin{cases}
		(-\De)^\frac12 u+V_0 u=g(v),~\text{in}~\mb R,\\
		(-\De)^\frac12 v+V_0 v=f(u),~\text{in}~\mb R.
	\end{cases}
\end{equation*}
\item $u_\e(x+x_\e/\e),~v_\e(x+x_\e/\e)\ra0$, uniformly as $|x| \ra 0$, for $\e>0$ sufficiently small.
	\end{enumerate}
\end{Proposition}
The second one states as follows
\begin{Proposition}
	Let $(\varphi_\e, \psi_\e)$ be a ground-state solution to \eqref{1.1} and $x_\e^1$, $x_\e^2$ be any maximum point of $|\varphi_\e|$ and $|\psi_\e|$ respectively. Then
	\begin{equation*}
		\lim\limits_{\e \ra 0} \text{dist}(x_\e^k,\mc L)=0,~\lim\limits_{\e \ra0} |x_\e^k-x_\e|=0,~k=1,2.
	\end{equation*}
\end{Proposition}
\begin{proof}
	Note that $x^1_\e/\e$,  $x^2_\e/\e$ are the maximum points of $u_\e$, $v_\e$, respectively. Thanks to the decay of $u_\e$, $v_\e$ and the following fact
	\begin{equation*}
		\liminf\limits_{\e \ra 0} \min\{\|u_\e\|_\infty,\|v_\e\|_\infty\}>0,
	\end{equation*}
we get $|x^1_\e/\e-x_\e/\e|$ is bounded for $k=1,2$ and $\e >0$ small enough. Then, $\ds \lim_{\e \ra 0} \text{dist}(x^k_\e,\mc L)=0$, $k=1,2$, $\lim_{\e \ra0}|x^k_\e-x_\e|=0$, $k=1,2$ and $\lim_{\e \ra0}|x^1_\e-x^2_\e|=0$. \QED
	\end{proof}

	\end{document}